\documentclass[reqno,final,11pt]{amsart}
\usepackage{natbib}  %Nature-like bibliography
\usepackage{fancyhdr} %Headers
\usepackage{color} %Color definition
\usepackage{hyperref} %Internal and external links
\usepackage{graphicx} %Graphics inclusion
\usepackage{geometry}
\usepackage[utf8]{inputenc}
\usepackage[T1]{fontenc}
\usepackage{verbatim}

\usepackage{tikz}
\usetikzlibrary{calc,shapes,backgrounds,arrows}
\usepackage{amssymb}
\usepackage[shortlabels]{enumitem}

\definecolor{aleacolor}{rgb}{0.16,0.59,0.78}

\hypersetup{
breaklinks,
colorlinks=true,
linkcolor=aleacolor,
urlcolor=aleacolor,
citecolor=aleacolor}

\renewcommand{\cite}{\citet}

\theoremstyle{plain}
\newtheorem{theorem}{Theorem}[section]                                          
                          
\newtheorem{lemma}[theorem]{Lemma}
\newtheorem{corollary}[theorem]{Corollary}

\theoremstyle{definition}

\theoremstyle{remark}
\newtheorem{remark}[theorem]{Remark}

\makeatletter \@addtoreset{equation}{section} \makeatother
\renewcommand\theequation{\thesection.\arabic{equation}}
\renewcommand\thefigure{\thesection.\arabic{figure}}
\renewcommand\thetable{\thesection.\arabic{table}}

\newcommand{\aleaIndex}[1]{\href{http://alea.impa.br/english/index_v22.htm}{\bf 22}}
\newcommand{\aleaDOI}[1]{\href{https://doi.org/10.30757/ALEA.v22-10}{DOI: 10.30757/ALEA.v22-10}}
\eheader{ALEA, Lat. Am. J. Probab. Math. Stat.}{\aleaIndex{22}}{2025}{245}{282}{\aleaDOI}
\elogo{\parbox[c]{180pt}{\includegraphics[width=170pt]{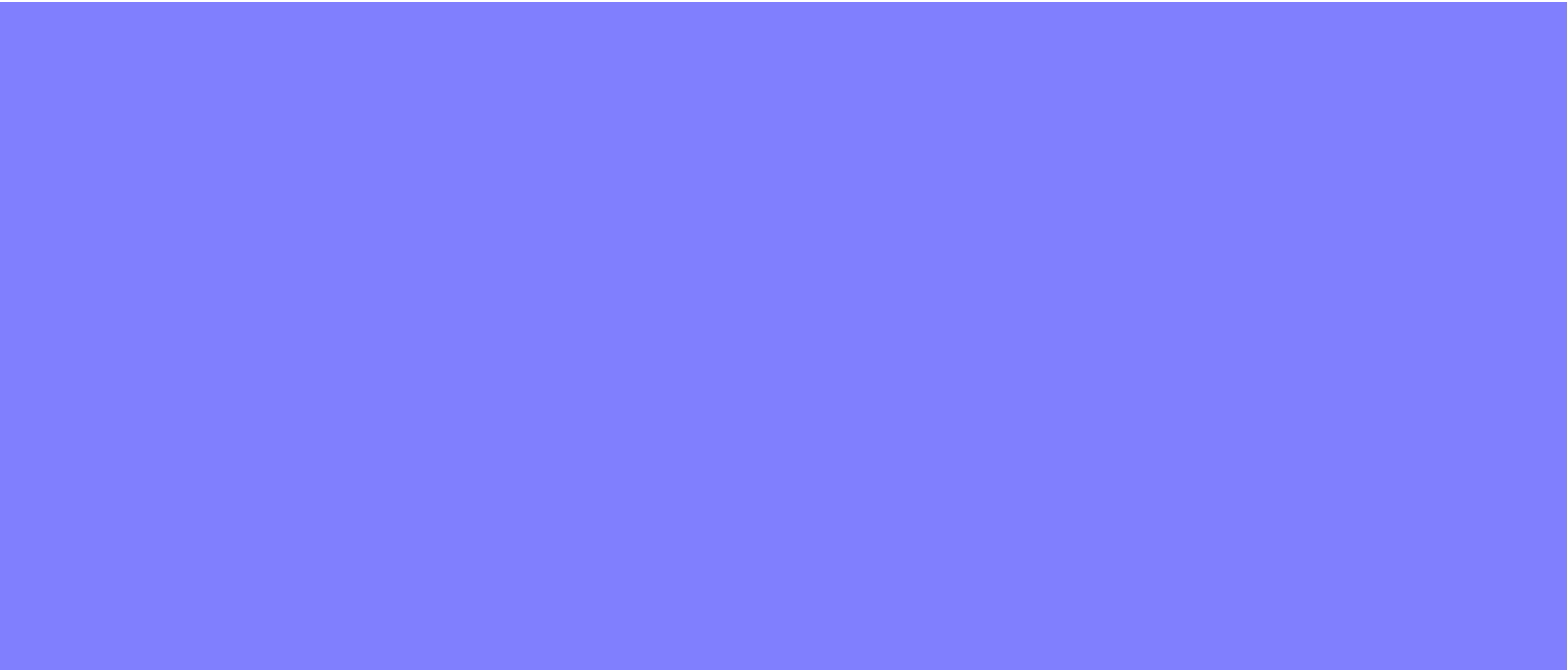}}} % 3 centimeters = 85.0393701 PostScript points

\newcommand{\abs}[1]{\left\vert#1\right\vert}
\newcommand{\norm}[1]{\left\Vert#1\right\Vert_2}
\renewcommand{\hat}{\widehat}

\def\CP{\xrightarrow{\P}}

\def\s{\hat{s}}
\def\w{\hat{w}}

\def\i{\hat{i}}
\def\ep{\epsilon}

\def\ER{Erd\H{o}s-R\'enyi}
\def\E{\mathbb{E}}
\def\P{\mathbb{P}}
\def\CM{\mathbb{CM}}

\begin{document}

\title{SIR Epidemics on Evolving Erd\H{o}s-R\'enyi Graphs}

\author{Wenze Chen}
\address{School of Mathematics and Statistics, Donghua University, Shanghai, China, 201620.}
\email{chenwenze@mail.dhu.edu.cn} 
% \urladdr{\href{http://First.com/~address}{http://first.com/$\sim$address}} 

\author{Yuewen Hou}
\address{School of Mathematics and Statistics, Fujian Normal University, Fuzhou, China, 350007.}
\email{853345720@qq.com} 

\author{Dong Yao}
\address{Research Institute of Mathematical Science, Jiangsu Normal University, Xuzhou, China, 221116.}
\email{dongyao@jsnu.edu.cn}

% \author{Second Author}
% \address{Second Author Adress}
% \email{second@email.com} 
% \urladdr{\href{http://second.com/~address}{http://second.com/$\sim$address}} 

% \author{Third Author}
% \address{Third Author Adress}
% \email{third@email.com} 
% \urladdr{\href{http://third.com/~address}{http://third.com/$\sim$address}} 

\thanks{Dong Yao is supported by NSFC grant (No. 12201256) and  NSF Jiangsu Province grant (No. BK20220677)}
\thanks{Dong Yao is the corresponding author of this work.}

\begin{abstract}
	In the standard SIR model, infected vertices infect their  neighbors at rate $\lambda$  independently across each edge. They also recover at rate $\gamma$. In this work we consider  the SIR-$\omega$ model where the graph structure itself co-evolves with the  SIR dynamics. Specifically,   $S-I$ connections  are broken at rate $\omega$. Then, with probability $\alpha$, $S$ rewires this edge to another uniformly chosen vertex; and with probability $1-\alpha$, this edge is simply dropped. When $\alpha=1$ the SIR-$\omega$ model becomes the evoSIR model introduced in  \cite{DOMath}, where they proved that the probability of a major outbreak in the evoSIR model evolving on an \ER($n$,$\mu/n$) graph converges to 0 as $\lambda$ approaches the critical infection rate $\lambda_c$. On the other hand, numerical experiments in \cite{DOMath}  revealed that, as  $\lambda \to \lambda_c$, (conditionally on a major outbreak) the fraction of infected vertices may not converge to $0$, leading to a \emph{discontinuous phase transition}. \cite{MR4456028} considered the SIR-$\omega$ model on an \ER($n$,$\mu/n$) graph and  proved   that, if
	\begin{equation*}
		\omega(2\alpha-1)>\gamma \mbox{ and } \mu>\frac{2\omega \alpha }{\omega(2\alpha-1)-\gamma},
	\end{equation*}
	then the phase transition of the final epidemic size is discontinuous. In this work we prove that the above  condition is also necessary for this discontinuity. We achieve this by proving that the time-changed SIR-$\omega$ process converges to a  singular ODE system. 
\end{abstract}

\maketitle

\section{Introduction}

In the SIR model, individuals are in one of  three states: $S=$ susceptible, $I=$ infected, $R=$ removed (cannot be infected).  In this work we consider the SIR model on a graph $G$ that can be thought of as the social structure of the underlying community: vertices represent individuals and edges connections between them. $S-I$ edges become $I-I$ at rate $\lambda$, i.e., after a time with an Exponential($\lambda$) distribution.  An individual remains infected for an amount of time which has an Exponential($\gamma$) distribution. Once individuals leave the infected state, they enter the removed state and stay removed forever. 

In recent years there has been a surge of interest in epidemic  processes on \emph{dynamic} networks \cite{MR4474536,MR4455874,MR4456028,MR4125785,MR4474532}, i.e., the  graph structure will not remain static but rather evolve with time.  \cite{MR3571300} introduced the \emph{SIR-$\omega$ model} where the graph is allowed to  co-evolve with the epidemic. In addition to the dynamics in the standard SIR model, in  the SIR-$\omega$ model $S-I$ edges are broken at rate $\omega$. With probability $1-\alpha$ the susceptible individual drops this edge, and with probability $\alpha$ $S$ re-connects to another individual chosen uniformly at random from the other $n-2$ individuals in the graph.  Mathematical analysis for this model  has been done in  \cite{MR4456028, MR3571300, LBSB}. When $\alpha=1$, the SIR-$\omega$ model is also called the \emph{evoSIR model},   where `evo' stands for `evolving'. The evoSIR model was studied in \cite{DOMath} and the authors there proved that,  similarly to the case of static graphs, as the infection rate $\lambda \searrow \lambda_c$, i.e., $\lambda$ approaches the critical value $\lambda_c$ from above, the probability of a major outbreak converges to 0. (For $\lambda <\lambda_c$, the probability of a major outbreak is 0, so we do not need to consider the limit of $\lambda \nearrow \lambda_c$.) In other words, the transition for the probability of a major outbreak is always continuous. However, as discovered  by simulations in \cite{DOMath}, the evoSIR model exhibits \emph{discontinuous phase transitions} in terms of the final epidemic size for certain values  of parameters of the model. More precisely, when  $\lambda$ approaches  $\lambda_c$, conditionally on a major outbreak the fraction of infected individuals \emph{doesn't} converge to 0. This is in sharp contrast with the case of SIR epidemics on static \ER\  graphs (and the configuration model), where the phase transition for the fraction of infecteds is always continuous. To avoid potential confusion,  when we say  `phase transition' below we always refer to the transition for the final epidemic size conditionally on a major outbreak.

While the SIR-$\omega$/evoSIR model may not be actually used in real  epidemic preventions, it is interesting in at least two aspects. First, the evoSIR model can increase the epidemic threshold (see \eqref{lcevosir}) while keeping the total number of edges in the network unchanged, which can be a bit surprising at first sight. Second, from the mathematical point view, apart from the phase transition phenomenon, as we will see below the SIR-$\omega$ model  provides an example of a particle system with a singular limit, which is interesting in itself. 

For the simpler \emph{SI} epidemic where individuals stay infected forever (i.e., no recoveries), \cite{MR4456028} and \cite{MR4474532} give an (almost) necessary and sufficient condition for discontinuous phase transitions of the final epidemic size in the \emph{SI-$\omega$ model} on \ER\  graphs and the \emph{evoSI} model on the configuration model, respectively. Note that \ER\  graphs  are  contiguous to the configuration model with Poisson degree distribution. In fact, \cite{MR4456028}  found out the limit of scaled final epidemic size, which can be expressed as a solution to a certain equation (see \cite[Theorem 2.5 and equation (2.8)]{MR4456028}). 

However, the situation becomes much more complicated for SIR epidemics on evolving graphs. The paper \cite{MR4456028} give two conditions for continuous and discontinuous phase transitions of the SIR-$\omega$ model. But there is a gap between these two conditions. In addition, not much is known about  the  final epidemic size. See Section \ref{sec:bb} for more details. 

In this work we  give a necessary and sufficient condition for the continuity of the phase transition of the SIR-$\omega$ model (Theorem \ref{phase}), resolving a conjecture  in \cite[Remark 2.4]{MR4456028}. For the case of continuous phase transitions we also prove the convergence of (scaled) final epidemic size (Theorem \ref{finalsize}).  A major step in the proof is to show  that the (time-changed version of) SIR-$\omega$ process converges to  the ODE system \eqref{hats}-\eqref{hatw}, which contains a discontinuity point. This barrier forbids direct  applications of ODE approximation techniques and causes many technicalities. 

The rest of the introduction is organized as follows. We state our main results in Section \ref{sec:mr}. Section \ref{sec:pw} is devoted to a quick review of related work. Section \ref{sec:ps} contains the proof strategies as well and the organization of this paper. 

\subsection{Main results}\label{sec:mr}

We start from some terminologies and definitions. The term \emph{final epidemic size}, denoted by $\Lambda^{(n)}=\Lambda^{(n)}(\alpha,\lambda,\gamma,\mu,\omega)$, when the graph has size $n$, refers to  the total number of individuals that have ever been  infected. For SIR epidemics it is the number of vertices that are eventually removed. We say a \emph{major outbreak} occurs if the epidemic infects more than $\ep n$ individuals ($n$ is the size of the total population) for some $\ep>0$ independent of $n$, i.e., $\Lambda^{(n)}\geq \ep n$. The critical infection rate $\lambda_c$ is the smallest infection rate such that a major outbreak occurs with probability bounded away from 0 as $n\to\infty$. Mathematically,
\begin{equation*}
  	\lambda_c=\inf\{\lambda>0: \liminf_{n\to\infty}\P(\Lambda^{(n)}\geq \ep n)>0 \mbox{ for some }\ep>0\}.
\end{equation*}
The phase transition for the final epidemic size is said to be continuous if for any $\ep>0$,
\begin{equation}\label{defconti}
  	\lim_{\lambda \to \lambda_c}\P(\Lambda^{(n)}>n\ep|\mbox{a major outbreak occurs})=0.
\end{equation}
If \eqref{defconti} fails for some $\ep_0>0$, then the phase transition is said to be discontinuous. 

As mentioned in the beginning of this paper, the SIR-$\omega$ model has four parameters: $\gamma$ representing the recovery rate, $\lambda$ the infection rate, $\omega$ the rate of dropping or rewiring and $\alpha$ the probability of rewiring. 
We focus on the \ER($n$,$\mu/n$) graph, which is a random graph on $n$ vertices where each pair gets connected with probability $\mu/n$.

It has been proved in \cite{DOMath} for the evoSIR model and \cite{MR4456028} for the SIR-$\omega$ model that 
\begin{equation}\label{crit_lambda}
	\lambda_c=\frac{\omega+\gamma}{\mu-1}.
\end{equation}

Turning to the phase transition of the final epidemic size, Ball and Britton proved in \cite[Theorem 2.4]{MR4456028} that if
\begin{equation}\label{cond_discont}
	\omega(2\alpha-1)>\gamma \mbox{ and } \mu>\frac{2\omega \alpha }{\omega(2\alpha-1)-\gamma},
\end{equation}
then there exists $r_0=r_0(\mu, \gamma,\omega,\alpha)>0$ such that
\begin{equation}\label{discont}
	\lim_{n\to\infty}\P(\Lambda^{(n)}>nr_0|\mbox{a major  outbreak occurs})=1 \mbox{ for all }\lambda>\lambda_c.
\end{equation}
They also showed that the phase transition is continuous when 
\begin{equation}\label{bbcond_cont}
	\omega(3\alpha-1)\leq \gamma \mbox{ or }\mu\leq \frac{2\omega \alpha }{\omega(3\alpha-1)-\gamma}. 
\end{equation}
Note that there is a gap between \eqref{cond_discont} and \eqref{bbcond_cont}. They conjectured in \cite[Remark 2.4]{MR4456028} that \eqref{cond_discont} should be both necessary and sufficient for a discontinuous phase transition. The following theorem of our work confirms this conjecture. 

\begin{theorem}\label{phase}
	Consider the SIR-$\omega$ model starting with one infected individual and the rest being susceptible. 
	Suppose
	\begin{equation}\label{cond_cont}
		\omega(2\alpha-1)\leq \gamma \mbox{ or }\mu\leq \frac{2\omega \alpha }{\omega(2\alpha-1)-\gamma},
	\end{equation}
	then for any $\ep>0$, there exists  $\lambda_1>\lambda_c$ such that
	\begin{equation}\label{cont}
		\lim_{n\to\infty}\P(\Lambda^{(n)}<\ep n|\mbox{a major outbreak occurs})=1 \mbox{ for all }\lambda\in (\lambda_c,\lambda_1).
	\end{equation}
	Combining \eqref{cont} and \eqref{discont}, the  phase transition for the final epidemic size of the SIR-$\omega$ model  near $\lambda_c$ is discontinuous if and only if \eqref{cond_discont} holds. 
\end{theorem}

We emphasize that, while the gap  between \eqref{cond_discont} and \eqref{bbcond_cont} may seem to be not very large, overcoming it  requires new ideas that are quite different from \cite{MR4456028}. To start we re-use the infrastructure built in \cite{MR4456028}.  Suppose the underlying graph has $n$ vertices, then we set
\begin{itemize}
	\item $S^{(n)}(t)$: the number of susceptible vertices at time $t$,
	\item $I^{(n)}(t)$: the number of infected vertices at time $t$,
	\item $I^{(n)}_E(t)$: the number of infectious  (i.e., infected--susceptible) edges at time $t$,
	\item $\tau^{(n)}$: the first time when $I_E^{(n)}$ reaches 0,
	\item $W^{(n)}(t)$: the number of susceptible--susceptible edges at time $t$ that were created by rewiring,
	\item ${\mathbf X}^{(n)}(t)=(S^{(n)}(t),I^{(n)}(t),I^{(n)}_E(t),W^{(n)}(t))$.
\end{itemize}
Note that $\tau^{(n)}$ can be understood as the \emph{terminal time} since no more vertices can be infected after $\tau^{(n)}$. Also, recall that we use  $\Lambda^{(n)}$ to denote the final epidemic size, which is equal to $n-S^{(n)}(\tau^{(n)})$. As will be clear soon, ${\mathbf X}^{(n)}$ is more amenable to analysis after a \emph{random time-transform}. Specifically, we  multiply all transition rates of the SIR-$\omega$ process  (which is a Markov chain) by $n/ (\lambda I_E^{(n)}(t))$. We use a hat to denote the  quantities in the time-changed process. For instance, $\hat{\tau}^{(n)}$ is the terminal time in the time-transformed model. Note that  ${\mathbf{X}}^{(n)}(0)=\hat{{\mathbf{X}}}^{(n)}(0)$. Also,  $\hat{\Lambda}^{(n)}$ is the final epidemic size in the time-changed model, which is necessarily equal to $\Lambda^{(n)}$. We emphasize that \emph{these notations  are fixed throughout the paper}.

Our goal is to give a deterministic approximation for $\hat{\mathbf X}(t)$. Consider the following system of equations: 
\begin{align}
	\frac{\mathrm{d}\s}{\mathrm{d}t}&=-1, \label{hats} \\
	\frac{\mathrm{d}\i}{\mathrm{d}t}&=-\frac{\gamma}{\lambda} \frac{\i}{\i_E}+1, \label{hati}\\
	\frac{\mathrm{d} \i_E}{\mathrm{d}t}&=-1-\frac{\gamma}{\lambda}+\mu \s-\frac{\i_E}{\hat{s}}  +2 \frac{\w}{\s}-\frac{\omega}{\lambda}  (1-\alpha +\alpha (1-\i)), \label{hati_E}\\
	\frac{\mathrm{d}\w}{\mathrm{d}t}&=\frac{\omega}{\lambda} \alpha  \s-2 \frac{\w}{\s}. \label{hatw}
\end{align}
Here $\hat{i}_E$ and $\hat{w}$ are non-negative functions, while $\hat{s}$ are $\hat{i}$ are non-negative and upper bounded by $1$. For $\hat{s}$ and $\hat{w}$, they are well defined up to time $t=\hat{s}(0)$; for $\hat{i}$ and $\hat{i}_E$, the appropriate domain is $t<t_*$ (see \eqref{deft_*} below). As will be clear later, these equations can be obtained by dividing the right-hand sides of equations \eqref{0s}-\eqref{0w} by $\lambda i_{E}(t)$. In fact,  Ball and Britton also mentioned this system of equations in \cite[Section 6.5.1]{MR4456028}, but they didn't make use of this to analyze the SIR-$\omega$ model. Note that \eqref{hati} contains a singularity when $\hat{i}_E=0$. \cite{MR4456028} got around the singularity issue by analyzing other auxiliary models.

Note that $(\hat{s}(t),\hat{w}(t))$ already forms a closed sub-system, which can be explicitly solved for any initial condition with $\hat{w}(0)=0$ as follows (for all $t<\hat{s}(0)$)
\begin{align}
	\hat{s}(t)&=\hat{s}(0)-t, \label{ssolve}\\
	\hat{w}(t)&=\frac{\omega \alpha}{\lambda}(\hat{s}(0)-t)^2 \log \frac{\hat{s}(0)}{\hat{s}(0)-t}. \label{wsolve}
\end{align}
We will show in Theorem \ref{exiuni} below  that the system \eqref{hats}-\eqref{hatw} admits a unique solution   under initial conditions that are relevant to the evolution of the  SIR-$\omega$ model.
\begin{itemize}
	\item \emph{Positive initial condition}: 
	\begin{equation*}
		\hat{i}(0)\in (0,1),\  \hat{s}(0)\in (0, 1-\hat{i}(0)),\ \hat{i}_E(0)>0, \ \hat{w}(0)=0.
	\end{equation*}
	\item \emph{Zero initial condition}: 
	\begin{equation*}
		\hat{s}(0)=1,  \ \hat{i}(0)=\hat{i}_E(0)=0, \ \hat{w}(0)=0.
	\end{equation*}
\end{itemize}
The positive initial condition corresponds to a positive fraction of individuals initially infected  while the zero initial condition is useful for analyzing  the case of one initially infected node. 

We denote the solution to \eqref{hats}-\eqref{hatw}   by $\hat{\mathbf{x}}(t):=(\hat{s}(t),\hat{i}(t),\hat{i}_E(t),\hat{w}(t))$. 
For any solution  $\hat{{\mathbf  x}}(t)$, we let $t_*$ to be the first time after 0 when $\hat{i}_E(t)$ hits 0, i.e., 
\begin{equation}\label{deft_*}
	t_*:=\inf\{0< t\leq \hat{s}(0): \hat{i}_E(t)=0\}.
\end{equation}
Here we use the convention that $\inf \emptyset=\hat{s}(0)$. We shall show in Section \ref{sec:eq=0} that $t_*<\hat{s}(0)$. 

We now state the existence and uniqueness theorem up to time $t_*$.  

\begin{theorem}\label{exiuni}
	For either positive initial condition or zero initial condition (in this case assuming additionally that $\lambda>\lambda_c$), there exists a unique solution $\hat{\mathbf{x}}(t)$ to the system \eqref{hats} -\eqref{hatw} up to time $t_*$.
\end{theorem}

With Theorem \ref{exiuni} established, we can now state the convergence theorem for  the stochastic process $\hat{\mathbf {X}}/n$. 

\begin{theorem}\label{limitode}
	For either positive initial condition or zero initial condition, suppose
  \begin{equation*}
    \hat{\mathbf{X}}^{(n)}(0)/n \CP \hat{\mathbf{x}}(0),
  \end{equation*}
	then 
	\begin{equation}\label{eq:hatsw}
		\sup_{t\leq \hat{s}(0) \wedge \hat{\tau}^{(n)}} \left( \abs{\frac{1}{n}\hat{S}^{(n)}(t)-\hat{s}(t) } + \abs{\frac{1}{n}\hat{W}^{(n)}(t)-\hat{w}(t) } \right) \CP 0,
	\end{equation}
	and 
	\begin{equation}\label{eq:hatx}
		\sup_{t\leq t_* \wedge \hat{\tau}^{(n)}} \norm{\frac{1}{n}\hat{\mathbf{X}}^{(n)}(t)-\hat{\mathbf{x}}(t) }\CP 0.
	\end{equation}
\end{theorem}

Theorem \ref{limitode} can be used to derive the convergence of the scaled  final epidemic size $\Lambda^{(n)}/n$ under \eqref{cond_cont}.

\begin{theorem}\label{finalsize}
	Assume that \eqref{cond_cont} holds. Let $t_*$ be given by \eqref{deft_*}. 
	\begin{enumerate}[label=(\roman*)]
		\item For positive initial condition:	suppose $ \hat{\mathbf{X}}^{(n)}(0)/n\CP \hat{\mathbf{x}}(0)$ where $\hat{i}(0)>0$ and $\hat{i}_E(0)>0$. Then, 
		\begin{equation}\label{1-st*}
			\frac{\Lambda^{(n)}}{n}\CP 1-\hat{s}(0)+t_* \quad \mbox{as} \quad n\to\infty.
		\end{equation}
		\item For zero initial condition: suppose $\lambda>\lambda_c$, and for each $n$ the epidemic starts from one infected vertex with the rest being susceptible.
		Then conditionally on a major outbreak,    \begin{equation}\label{final1}
			\frac{\Lambda^{(n)}}{n}\CP t_* \quad \mbox{as}\quad  n\to\infty.
		\end{equation}
	\end{enumerate}
\end{theorem}

\begin{remark}
	Note that the right-hand sides of \eqref{1-st*} and \eqref{final1} can be written as $1-\hat{s}(t_*)$. 
	Theorem \ref{finalsize} thus verifies \cite[Conjecture 2.1]{MR4456028} under the condition \eqref{cond_cont}. We conjecture that this condition is not essential, i.e.,  Theorem \ref{finalsize}  holds without assuming \eqref{cond_cont}.  
\end{remark}

Theorem \ref{phase} follows from Theorem \ref{finalsize} via a direct computation. See Section \ref{sec:contphase}. 

We end this section with a comment regarding the SIR-$\omega$ model on random graphs generated from the configuration model. 

\begin{remark}
  	While we believe  that a similar phase transition should occur for the SIR-$\omega$ model on the configuration model,  we do find  some nontrivial barriers in adapting the proof of Theorem \ref{phase} to that case. See also Remark \ref{rem:dif} below. 
\end{remark}

\subsection{Previous work}\label{sec:pw}

\subsubsection{SIR on static graphs}

So far the SIR model on several random graph models (including \ER\  graphs and the configuration model) has been well understood.  For the case of \ER($n$,$\mu/n$) graphs,  let $T$ be a random variable with the  Exponential($\lambda$) distribution, and consider the functions
\begin{equation*}
  	G_0(z)=\mathbb{E}\exp\left(-\mu (1-z)(1-e^{-\lambda T})\right), \quad G_1(z)=\exp\left(-\mu \frac{\lambda}{\lambda+\gamma} (1-z)\right).
\end{equation*}
The following theorem is well known. See, e.g., \cite{MR1993267}.  

\begin{theorem}
	The critical value for SIR model on static \ER($n$,$\mu/n$)  is given by $\lambda_c=\gamma/(\mu-1)$.  For $\lambda>\lambda_c$,
	the probability of a major outbreak converges to   $1-z_0$ where $z_0<1$ is the solution to the equation $G_0(z)=z$, i.e., 
	$$
	\lim_{n\to\infty}\P(\Lambda^{(n)}\geq \ep n) = 1-z_0
	$$
	for all small enough $\ep>0$. 
	In addition, conditionally on a major outbreak,  the scaled final epidemic size $\Lambda^{(n)}/n$  converges to $1-z_1$, where $z_1<1$ is the solution to the equation $G_1(z)=z$. In particular, $z_0\neq z_1$. 
\end{theorem}

It follows from this theorem that the phase transition for final epidemic size of the  SIR model on static \ER\  graphs is continuous, since we have $1-z_1\to 0$ as $\lambda\to \lambda_c$.

Similar results hold for the configuration model. Take any degree sequence $(D_1,\ldots, D_n)$, denoted by $\mathbf{D}_n$, the configuration model $ \CM(n,\mathbf{D}_n)$ is constructed by uniformly randomly pairing $D_1,\ldots,D_n$ half-edges attached to vertices $1,\ldots, n$, respectively. Two paired half-edges form a single edge.  If $D_i, 1\leq i\leq n$, are i.i.d. sampled from a distribution $D$ (with minor adjustment to make sure that the sum $\sum_{i=1}^n D_i$ is even), then we may write $\CM(n,D)$ for $\CM(n,\mathbf{D}_n)$.  It is known that, when $\E(D^2)<\infty$, the local neighborhood of $\CM(n,D)$ resembles a two-stage Galton-Watson tree where the offspring distribution of the root is $D$ while the offspring distribution of later generations is  $D^*-1$ ($D^*$ is the size-biased version of $D$). See, e.g., \cite[Chapter 7]{MR3617364} for a detailed introduction to the configuration model.  Then for the SIR model on $\CM(n,D)$, $\lambda_c=\gamma m_1/(m_2-2m_1)$ where $m_i$ is the $i$-th moment of $D$. In addition,  both the probability of a major outbreak and the scaled final epidemic size (conditionally on a major outbreak) converge to a deterministic quantity which  goes to 0 as $\lambda \searrow\lambda_c$. See, e.g., \cite{MR3275704} for statements and proofs of these results. 

\subsubsection{SI on evolving graphs}

As mentioned before, in SI type models infected vertices will not recover. 

For the  SI-$\omega$ epidemics on \ER($n,\mu/n$) graphs, \cite{MR4456028} found out the limit of $\Lambda^{(n)}/n$ (which can be expressed as the solution to an explicit equation) and gave a necessary and sufficient condition for the phase transition to be discontinuous. See \cite[Theorem 2.5 and Theorem 2.6]{MR4456028}. More precisely, they showed that the phase transition for the fraction of infected is discontinuous if and only if $\alpha>1/3$ and $\mu>3\alpha/(3\alpha-1)$. 

When $\alpha=1$, the SI-$\omega$ model is also called the evoSI model. \cite{MR4474532} analyzed the phase transition of final epidemic size in the  evoSI model with underlying graph $\CM(n,D)$. Let $\mu_i$ be the $i$-th factorial moment of $D$, i.e., 
\begin{equation*}
  	\mu_i=\E[D (D-1)\cdots (D-i+1)].
\end{equation*}
Set
\begin{equation}
  	\Delta= -\frac{\mu_3}{\mu_1} + 3 (\mu_2-\mu_1).\label{deldef}
\end{equation}
They showed that the phase transition is continuous if $\Delta<0$ and discontinuous if $\Delta>0$. Note that when $D\sim \mathrm{Poisson}(\mu)$ and $\alpha=1$,  $\Delta=-\mu^2+3(\mu^2-\mu)$. Thus, the condition $\Delta>0$ corresponds to $\mu>3/2$, which is the condition given by Ball and Britton by setting $\alpha=1$.

\subsubsection{SIR on evolving graphs}\label{sec:bb}

\cite{DOMath} studied the evoSIR model (SIR-$\omega$ with $\alpha=1$) on \ER($n$,$\mu/n$) graphs. They showed in \cite[Theorem 1]{DOMath} that the critical infection rate for evoSIR model is give by 
\begin{equation}\label{lcevosir}
	\lambda_c=\frac{\gamma+\omega}{\mu-1}.
\end{equation}
In addition, the probability of a major outbreak converges to  $1-z''_0$ where $z''_0<1$ is the solution to the equation $G_2(z)=z$ with $G_2$ defined  by
\begin{equation*}
  	G_2(z)=\exp\left(-\mu \frac{\lambda}{\lambda+\omega+\gamma} (1-z)\right).
\end{equation*}
Consequently, as $\lambda$ approaches $\lambda_c$, the probability of a major outbreak also converges to 0.
The authors there also used numerical experiments to discover  that for certain parameters the phase transition for final epidemic size is discontinuous. See Figure \ref{fig:ctvaryrho}.
\begin{figure}[h!]
	\centering
	\includegraphics[width=0.6\textwidth]{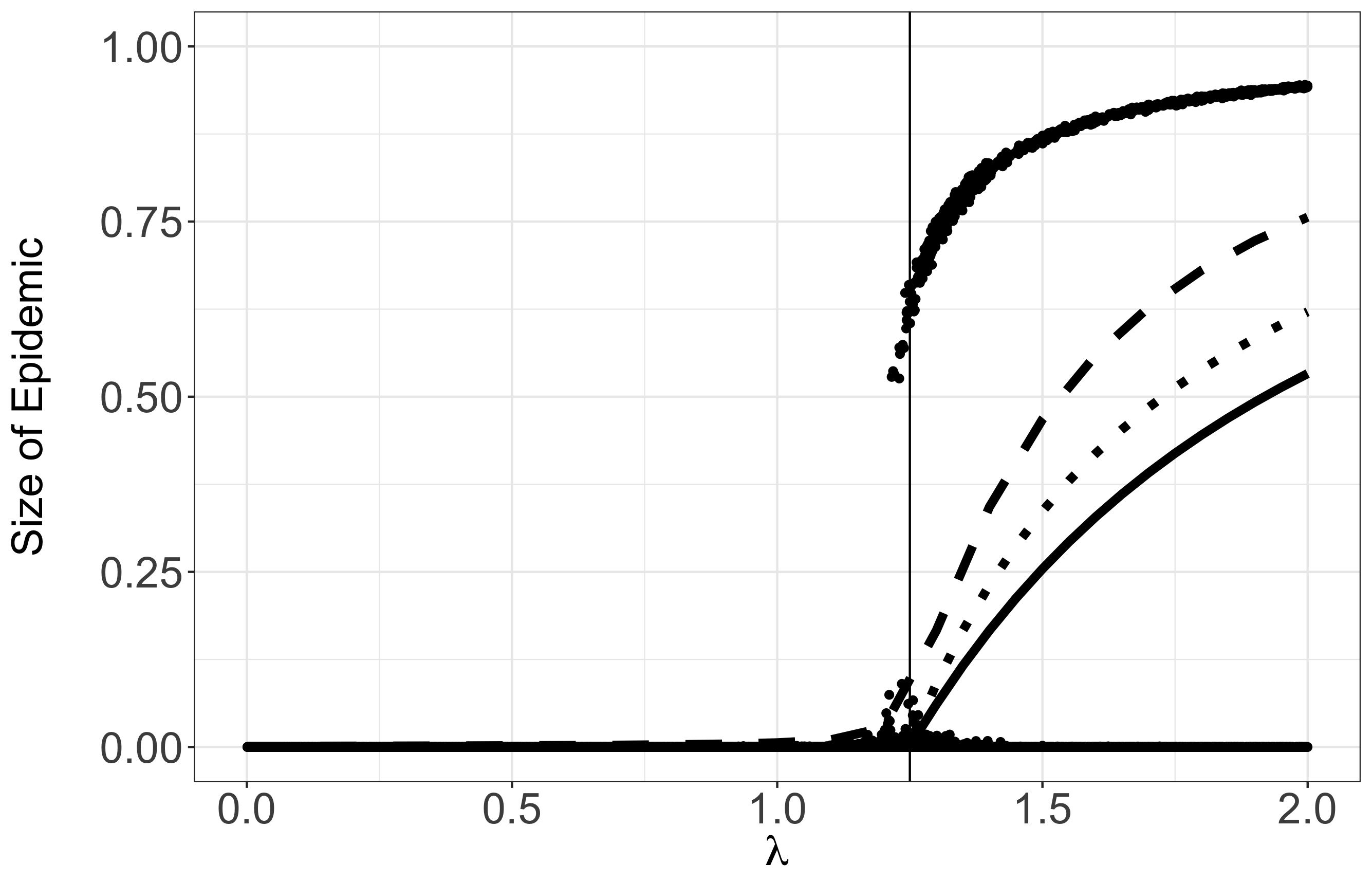}
	\caption{Simulation of the final epidemic size of the evoSIR model on an Erd\H os-R\'enyi graph with $\mu=5$, $\omega=4$ and $n\in [10^4,10^5]$. In this case $\lambda_c=1.25$. The top curve is the simulation of final size of evoSIR. The bottom curve is the final size of the delSIR (SIR-$\omega$ with $\alpha=0$) epidemic with the same parameters. Other two curves are approximations of evoSIR,  which are discussed in \cite{DOMath}.}\label{fig:ctvaryrho}
\end{figure}

SIR-$\omega$ model is introduced the  in \cite{MR3571300}. \cite{MR4456028} analyzed the SIR-$\omega$ model on an \ER($n$,$\mu/n$) graph using a novel construction that couples the graph and the epidemic. We will explain the construction in  Section \ref{s:construct}. Using this construction, they proved the convergence of $\mathbf{X}^{(n)}(t)/n$  to a system of ODEs over any finite time horizon. Specifically, let $\mathbf{x}(t)=(s(t),i(t),i_E(t),w(t))$ be the solution of the following system of equations:
\begin{align}
	\frac{\mathrm{d}s}{\mathrm{d}t}&=-\lambda i_{E}, \label{0s}\\
	\frac{\mathrm{d}i}{\mathrm{d}t}&=-\gamma i+\lambda i_E, \label{oi}\\
	\frac{\mathrm{d} i_E}{\mathrm{d}t}&=-\lambda i_E -\gamma i_E+\lambda \mu i_E s -\lambda \frac{i_E^2}{s}+2\lambda i_E \frac{w}{s}-\omega i_E (1-\alpha +\alpha (1-i)), \label{oie}\\
	\frac{\mathrm{d}w}{\mathrm{d}t}&=\omega \alpha i_E s-2\lambda i_E \frac{w}{s} \label{0w}.
\end{align}
Ball and Britton proved the following result in \cite[Theorem 2.2]{MR4456028}. 

\begin{theorem}
	Suppose $  {\mathbf X}^{(n)}(0)/n\CP \mathbf{x}(0)$ where $i(0)>0$ and $i_E(0)>0$. Then, for any $t_0>0$, 
	\begin{equation*}
		\sup_{0\leq t\leq t_0} \norm{ \frac{{\mathbf{X}}^{(n)}(t)}{n} -\mathbf{x}(t)}\CP 0 \mbox{ as }n\to\infty. 
	\end{equation*}
\end{theorem}

Note that this theorem doesn't tell us what happens when the time is $\infty$, thus we cannot use it to deduce the final epidemic size. Ball and Britton  conjectured \cite[Conjecture 2.1]{MR4456028}  that the limit of $n\to\infty$ and $t\to\infty$  can be interchanged so that $\Lambda^{(n)}/n \to 1-s(\infty)$. Our Theorem \ref{finalsize} confirms this conjecture under the condition \eqref{cond_cont}. 

By constructing models that are upper and lower bounds for the SIR-$\omega$ model, Ball and Britton proved in \cite[Theorem 2.4]{MR4456028} that the final epidemic size has a discontinuous phase transition under \eqref{cond_discont} and a continuous transition under \eqref{bbcond_cont}. These two conditions  have a gap.  We will close this gap by directly analyzing the SIR-$\omega$ model itself. 

\subsection{Proof strategies and organization of the paper}\label{sec:ps}

Our proof relies on the construction of the SIR-$\omega$ model given in \cite{MR4456028}. We will give the details in the next section. Using that construction we can derive evolution equations for quantities like $S^{(n)}(t)$. As pointed out in \cite{MR4456028}, a common approach to analyze epidemic models is performing a suitable time-transform so that the system becomes more solvable. In the case of the SIR-$\omega$ model, after the time-change, the limiting system of ODEs change from \eqref{0s}-\eqref{0w} to \eqref{hats}-\eqref{hatw}.   Note that the equation for $\hat{i}'$ involves the term $\hat{i}/\hat{i}_E$, which is not Lipschitz for $\hat{i}_E$ near 0. Thus, one cannot directly apply classical results (e.g., \cite{MR2395153}) to obtain the  convergence of the terminal time $\hat{\tau}^{(n)}$ and the fraction of infected vertices $\Lambda^{(n)}/n$. 

We now give a quick overview of the proof of \eqref{eq:hatx} in the case of positive initial condition to illustrate the idea to overcome this barrier. Fix any $\ep>0$, we  can divide the interval $[0,t_*]$ into $[0,t_*-\ep]$ and $[t_*-\ep, t_*]$. We note that $\hat{i}_E$ is bounded away from 0 in $[0,t_*-\ep]$. Consequently, standard arguments (tightness plus the uniqueness of the limit) can be used to derive \eqref{eq:hatx} for $t\leq t_*-\ep$. To deal with the other time interval $[t_*-\ep, t_*]$, we need two observations:
\begin{itemize}
	\item $\hat{i}(t_*-\ep)$ and $\hat{i}_E(t_*-\ep)$ can be arbitrarily small by taking $\ep$ to be small. It follows that $\hat{I}^{(n)}(t_*-\ep)/n$ and $\hat{I}^{(n)}_E(t_*-\ep)/n$ must also be small with high probability, by the convergence in the interval $[0,t_*-\ep]$ that we just proved. 
	\item The evolution equations for $\hat{I}^{(n)}$  and $\hat{I}^{(n)}_E$, \eqref{hatieq} and \eqref{hatiEeq}, imply that the upward drift for $\hat{I}^{(n)}$ and $\hat{I}^{(n)}_E$ is not large. 
\end{itemize}
Combining these two observations, we conclude that $\hat{I}^{(n)}/n$ and $\hat{I}^{(n)}_E/n$ remain small in the time interval $[t_*-\ep, t_*]$. Since $\hat{i}$ and $\hat{i}_E$ are also small in this interval,  we finish the proof of \eqref{eq:hatx}. 

In order to obtain the convergence of $\hat{\tau}^{(n)}$ and prove Theorem \ref{finalsize},  it is desirable to know whether the derivative of $\hat{i}_E$ is negative or 0 at $t_*$. However, this seems not trivial. To overcome this issue, we just consider the two cases $i'_E(t_*)<0$ and $i'_E(t_*)=0$ separately. The first case is relatively easy. In the second case we need more careful estimates and this is where we need the condition \eqref{cond_cont}. 

To prove Theorem \ref{phase} we also need to separate the condition \eqref{cond_cont} into two sub-cases according to whether the equality in \eqref{cond_cont} holds or not.

The rest of the paper is organized as follows.  An alternative construction of the SIR-$\omega$ model and several consequences of this construction  are given in Section \ref{s:construct}. We prove the existence and uniqueness of the solution to the system of ODEs \eqref{hats}-\eqref{hatw} (Theorem \ref{exiuni}) in Section \ref{sec:exiuni}. The convergence of the time-changed SIR-$\omega$ model to differential equations (Theorem \ref{limitode}) is proved in Section \ref{sec:conv}. We then use Theorem \ref{limitode} to prove the limit of the scaled final epidemic size  under the condition \eqref{cond_cont}  (Theorem \ref{finalsize}) in Section \ref{sec:finalsize}. Finally, in Section \ref{sec:contphase}, we use Theorem \ref{finalsize} to prove the continuity of the phase transition under the condition \eqref{cond_cont} (Theorem \ref{phase}).

\section{Construction of the SIR-$\omega$ process}\label{s:construct}

The key to our proof is the following construction for the SIR-$\omega$ model on \ER\  graphs, proposed in \cite{MR4456028}. Initially, there are $I^{(n)}(0)$ infected vertices and $I_E^{(n)}(0)$ edges attached to these vertices.   For each of these edges  it has a `free' endpoint which is undetermined at this stage. No edges are attached to other vertices. There are three types of Poisson process occurring on vertices or edges. (When there is an arrival in the Poisson process, we also say that Poisson clock rings.) Below we use the term \emph{infected edges} to denote the edges that are connected to infected vertices and have a free endpoint.  We also let $\mathcal{W}^{(n)}(t)$ be the pool of rewired susceptible edges, i.e., edges that have been rewired at least once and are  susceptible at time $t$. Note that if an edge in the pool $\mathcal{W}^{(n)}(t-)$  becomes an infected edge at time $t$, then it will exit $\mathcal{W}^{(n)}(t)$, though it may joint the pool again at some later time.
\begin{itemize}
	\item \textbf{Recovery Poisson processes for each infected vertex with rate $\gamma$}. Each infected vertex recovers with rate $\gamma$. When the recovery clock rings, the infected vertex becomes recovered and all edges attached to this infected vertex are no longer infected.
	\item \textbf{Rewiring/dropping Poisson processes for each infected edge with rate $\omega$}. Each infected edge has an independent Poisson process  with rate $\omega$ on it. Whenever the clock rings, this edge is dropped with probability $1-\alpha$ and rewired to a uniformly chosen vertex excluding itself with probability $\alpha$. Suppose at time $t$ such a clock rings on an edge $e$ attached to vertex $x$ and that edge decides to be rewired. Then  vertex  $x$ loses $e$ at time $t$. 
	\begin{itemize}
		\item If $e$ is rewired to an infected vertex $y$, then the number of infected edges of $y$ increases by 1. This case occurs with probability $(I(t)-1)/(n-1)$.
		\item  If $e$ is rewired to a susceptible vertex (which we don't specify at this moment), then  $e$ is added to $\mathcal{W}^{(n)}(t)$  and its cardinality $W^{(n)}(t)=W^{(n)}(t-)+1$. 
		\item If $e$ is rewired to a recovered vertex, then the effect is similar to dropping.
	\end{itemize}
	\item \textbf{Infection Poisson processes for each infected edge with rate $\lambda$}. Each infected edge is connected to a (uniformly chosen) susceptible vertex with rate $\lambda$. Suppose infected edge $e$ is connected to susceptible vertex $y$. Then $y$ immediately becomes infected. We now attach a random number of edges to $y$ (which become infected edges immediately) given by 
	\begin{equation}\label{infpoi}
		\mbox{Poisson}\left(\frac{S^{(n)}(t) }{n}\mu\right)+\mbox{Binomial}\left(W^{(n)}(t-),\frac{2}{S^{(n)}(t-) } \right).
	\end{equation}
We also delete all infected edges (other than those newly added to $y$) with probability $1/S^{(n)}(t-)$, independently. 
\end{itemize}

To explain the construction, note that we have required infected edges to be paired to susceptible vertices only. The definition of \ER\  graphs implies that  the number of  edges (excluding rewired edges) between $y$ and  susceptible vertices at time $t$ is close in distribution to Poisson($S^{(n)}(t)\mu/n$). On the other hand, the number of rewired  susceptible edges in the pool $\mathcal{W}^{(n)}(t-)$ that have an endpoint as  $y$  has a $\mbox{Binomial}(W^{(n)}(t-),2/S^{(n)}(t-)  )$ distribution. The reason of deleting infected edges with probability $1/S^{(n)}(t-)$ is to remove all infected edges that connect to  $y$ (so that all of the remaining infected edges connect to susceptible vertices). 

It has been proved in \cite[Section 2.3]{MR4456028} that the process just constructed (which we temporarily call the coupled SIR-$\omega$ model) has the same law as the original SIR-$\omega$ model on \ER($n$,$\mu/n$) graphs, after conditioning on the event
\begin{equation}\label{simple}
	\begin{split}
		\Omega_{\ref{simple}}:=\{& \mbox{The graph induced from the construction} \\
		& \mbox{of the couple SIR-$\omega$ model is a simple graph}\}.
	\end{split}
\end{equation}
In addition,  Ball and Britton showed that
\begin{equation*}
	\liminf_{n\to\infty}\P(\Omega_{\ref{simple}})>0.
\end{equation*}
Set
\begin{equation}\label{simple0}
	\Omega_{\ref{simple0}}:=\{\mbox{A major outbreak occurs in the coupled SIR-$\omega$ model}\}.
\end{equation}
Using the techniques in the proof of \cite[Theorems 2.7 and 2.10]{MR3275704}, we can show that
\begin{equation}\label{simple2}
	\liminf_{n\to\infty}\P(\Omega_{\ref{simple0}}\cap \Omega_{\ref{simple}})>0.
\end{equation}
Equation \eqref{simple2} implies that it suffices to prove Theorem \ref{phase} for the coupled SIR-$\omega$ model. To see this, we use a subscript `old'  to denote the original SIR-$\omega$ model and `new' to denote the coupled SIR-$\omega$ model. Suppose \eqref{cont} holds for the coupled model, then for $\lambda<\lambda_1$, 
\begin{equation}\label{simple3}
	\lim_{n\to\infty} \P(\Lambda^{(n)}_{\textrm{new}}>\ep n )=0.
\end{equation}
By \eqref{simple2} and \eqref{simple3},
\begin{equation*}
	\limsup_{n\to\infty} \P(\Lambda^{(n)}_{\textrm{new}}>\ep n | \Omega_{\ref{simple0}}\cap \Omega_{\ref{simple}} )\leq 
	\limsup_{n\to\infty} \frac{\P(\Lambda^{(n)}_{\textrm{new}}>\ep n )}{\P(\Omega_{\ref{simple0}}\cap \Omega_{\ref{simple}} )}=0.  
\end{equation*}
Since the original model and the coupled model are equivalent upon conditioning on $\Omega_{\ref{simple}}$, 
\begin{equation*}
	\lim_{n\to\infty} \P(\Lambda^{(n)}_{\textrm{old}}>\ep n | \Omega_{\ref{simple0}})=
	\lim_{n\to\infty} \P(\Lambda^{(n)}_{\textrm{new}}>\ep n | \Omega_{\ref{simple0}}\cap \Omega_{\ref{simple}} )=0.
\end{equation*}
This proves Theorem \ref{phase} for the original model. Hence, in this paper we will not distinguish between the original SIR-$\omega$ model and the coupled SIR-$\omega$ model. 

\begin{remark}\label{rem:dif}
  	his coupling construction of SIR-$\omega$ model is specific to \ER \  random graphs. The extension to the configuration model seems to be a difficult problem. 
\end{remark} 

In the rest of this section we will state several implications of the construction just given, which will be used in later sections. In the proofs of Lemmas \ref{l:numberedges} - \ref{boundni} below, it will be more convenient to work with the process in its original time scale. 

\begin{lemma}\label{l:numberedges}
	Let $A_e$ be the total number of edges appearing in the SIR-$\omega$ process. Then exists a constant $C>0$ such that
	\begin{equation*}
		\P(A_e\geq 2\mu n)\leq \frac{C}{n}.
	\end{equation*}
\end{lemma}

\begin{proof}
  	From the construction of the SIR-$\omega$ model,  $A_e$ is bounded from above the sum of two random variables 
	\begin{equation*}
    	I_E^{(n)}(0)+\sum_{i=1}^{n-I^{(n)}(0)} F_i  \mbox{ where } F_i \mbox{ i.i.d. } \sim \mbox{Poisson}(\mu).
  	\end{equation*}
	Clearly $I_E^{(n)}(0)$ is bounded from above by the sum of $I^{(n)}(0)n$ independent Bernoulli($\mu/n$) variables. Let $U_1$ have the distribution  Binomial($n^2$, $\mu/n$) and $U_2$ the distribution Binomial($n^2/2$, $\mu/n$).
	For the case $I^{(n)}(0)>n/2$,
	\begin{equation*}
    	A_e \mbox{ is stochastically dominated by } U_1+\sum_{i=1}^{n/2} F_i.
  	\end{equation*}
	For the case $I^{(n)}(0)\leq n/2$,
	\begin{equation*}
    	A_e \mbox{ is stochastically dominated by } U_2+\sum_{i=1}^{n} F_i.
  	\end{equation*}
	
	We have
	\begin{equation*}
		\E(U_1)=\mu n, \quad \textrm{Var}(U_1)\leq  \mu n.
	\end{equation*}
	Thus,
	\begin{equation}\label{number-1}
		\P(U_1\geq 5\mu n/4)=\P(U_1-\E(U_1)\geq \mu n/4)\leq \frac{\textrm{Var}(U_1)}{(\mu n/4)^2}\leq \frac{C}{n}.
	\end{equation}
	We also have
	\begin{equation}\label{number-2}
		\P\left(\sum_{i=1}^{n/2}F_i\geq 3 n\mu/4 \right)   =     \P\left(\sum_{i=1}^{n/2} (F_i-\mu) \geq n\mu/4 \right)\leq \frac{\mathrm{Var}(Y_1)n/2}{( n\mu/4)^2}
		\leq \frac{C}{n}.
	\end{equation}
	Combining \eqref{number-1} and \eqref{number-2},
	\begin{equation*}
		\begin{split}
			\P\left(A_e \geq 2\mu n, I^{(n)}(0)\geq n/2  \right)&\leq 
			\P\left( U_1+\sum_{i=1}^{n/2} F_i  \geq 2\mu n\right)\\
			&\leq     \P(U_1\geq 5\mu n/4)+  \P\left(\sum_{i=1}^{n/2}F_i\geq 3 n\mu/4 \right)   \\
			&\leq \frac{C}{n}+\frac{C}{n}\leq \frac{2C}{n}.
		\end{split}
	\end{equation*}
	
	This proves Lemma \ref{l:numberedges} for the case $I^{(n)}(0)\geq n/2$. The case $I^{(n)}(0)\leq n/2$ can be proved similarly. Combining these two cases we deduce Lemma \ref{l:numberedges}.
\end{proof}

By Lemma \ref{l:numberedges}, $\P(A_e\geq 2\mu n)\to 0$. From now on  we work on the event $\{A_e\leq 2\mu n\}$.

\begin{lemma}\label{bound of edge}
  	The number of times that any given edge is rewired is stochastically dominated by Geometric($\alpha \omega/(\omega+\lambda)$). Consequently, the number of rewired edges that any given vertex receives is stochastically dominated by Binomial($2\mu n$,$2\omega/(\lambda n)$). 
\end{lemma}

\begin{proof}
  	The rewiring/dropping   Poisson processes  and infection  Poisson processes occur independently for each edge in the construction of the SIR-$\omega$ model. By properties of Poisson processes, the probability that  rewiring/dropping occurs before infection is $\omega/(\omega+\lambda)$. Also note that the probability of rewiring  is $\alpha$ in the rewiring/dropping Poisson processes. As a result,  for any infected edge $e$,
	\begin{equation*}
		\P(\mbox{edge }e \mbox{ is rewired before  it infects other vertex or gets recovered})\leq \alpha\omega/(\omega+\lambda).
	\end{equation*}
  	To see this, first note that all other Poisson processes (other than the rewiring/dropping and infection Poisson processes associated with $e$) will either prevent $e$ from rewiring, or have no effect on its rewiring.  On the other hand, the probability that the rewiring Poisson clock rings first before dropping or infection Poisson clocks is equal to $\alpha \omega /(\omega+\lambda)$. 
  
  	Since whether each attempt to rewire succeeds or not is independent,  the number of times that any given edge is rewired is stochastically dominated by a Geometric distribution with success probability $\alpha \omega/(\omega+\lambda)$. On the other hand, when a rewiring occurs, the probability of any given vertex $x$  being selected to receive the rewired edge is at most $1/(n-1)$. As a result, the probability that any given edge $e$ has been   rewired to any given vertex $x$ is bounded from above by
	\begin{equation*}
		1-\E\left[\left( 1-\frac{1}{n-1}\right)^H   \right]                        \mbox{ where }H\sim \mbox{Geometric}(\alpha \omega/(\omega+\lambda)),
	\end{equation*}
	which is further bounded by (for large $n$)
	\begin{equation*}
		\E\left(\frac{H}{n-1}\right) = \frac{1}{n-1} \left(\left(1-\frac{\alpha \omega}{\omega+\lambda} \right)^{-1}-1 \right)\leq \frac{\omega}{(n-1)\lambda}\leq \frac{2\omega}{\lambda n}.
	\end{equation*}
	Since we are working on the event $\{A_e\leq 2\mu n\}$, we know that the number of  edges that have the potential to be rewired is less than or equal to $2\mu n$. Since the infection Poisson processes and rewiring/dropping Poisson processes are independent for different edges,  the total number of rewired edges that any given vertex receives is stochastically dominated by
	\begin{equation*}
    	\mbox{Binomial}\left(2\mu n, \frac{2\omega}{\lambda n}\right).
  	\end{equation*}
\end{proof}

\begin{lemma}\label{boundni}
	Let $N_i$ be the number of edges added to vertex $i$ when it first becomes infected.  Then $N_i$ is stochastically dominated by the independent sum
	\begin{equation*}
    	\mbox{Poisson}(\mu)+\mbox{Binomial}(2\mu n,\omega/(\lambda n)).
  	\end{equation*}
  	Consequently, for some constant $C_{\ref{ni4}}$ we have 
	\begin{equation}\label{ni4}
		\E(N_i^4)\leq C_{\ref{ni4}}.
	\end{equation}
	Similarly, setting $\hat{N}_i$ to be the number of edges of vertex $i$ when it becomes recovered, then
	\begin{equation}\label{hatni4}
		\E(\hat{N}_i^4)\leq C_{\ref{ni4}}.
	\end{equation}
\end{lemma}

\begin{proof}
	Let $\mathfrak{t}_i$ be the time when $i$ first becomes infected.
	Then $N_i$ is equal to the independent sum
  	\begin{equation*}
    	\mbox{Poisson}\left(\frac{S^{(n)}(\mathfrak{t}_i)}{n}\mu\right)+\mbox{Binomial}\left(W^{(n)}(\mathfrak{t}_i-),\frac{2}{S^{(n)}(\mathfrak{t}_i-) } \right),
  	\end{equation*}
	where the second part in the sum stands for the number of edges rewired to  vertex $i$ when $i$ first becomes infected. By Lemma \ref{bound of edge} we see $N_i$ is stochastically dominated by the independent sum Poisson($\mu$)+Binomial($2\mu n$,$2\omega/(\lambda n)$). Set $V\sim \mbox{Poisson}(\mu)$, $B\sim \mbox{Binomial}(2\mu n$,$2\omega/(\lambda n))$. In addition, $V$ and  $B$ are independent.  Then we have
	\begin{equation*}
		\E(N_i^4)\leq \E((V+B)^4).
	\end{equation*}
	Note that the fourth moment of Poisson($\mu$) is clearly finite, and 
	\begin{equation*}
		\E(B^4) \leq  \sum_{j=1}^4 (2\mu n)^{j}  \left(\frac{2\omega}{\lambda n}\right)^{j} \leq C.
	\end{equation*}
	Hence, for some sufficiently large constant $C_{\ref{ni4}}$,
	\begin{equation*}
		\E((V+B)^4) \leq \E(8(V^4+B^4))  \leq C_{\ref{ni4}},
	\end{equation*}
	which then implies that $\E(N_i^4)\leq C_{\ref{ni4}}$ and proves \eqref{ni4}. Equation \eqref{hatni4} can be proved in exactly the same way as \eqref{ni4} since $\hat{N}_i$ is also dominated by  the independent sum of  Poisson($\mu$) and Binomial($2\mu n$,$\omega/(\lambda n)$). 
\end{proof}

\section{Existence and uniqueness of the solution to \eqref{hats}-\eqref{hatw}}\label{sec:exiuni}

In this section we prove the existence and uniqueness of the solution to the system \eqref{hats}-\eqref{hatw}. Note that we have already given closed form expressions for $\hat{s}(t)$ and $\hat{w}(t)$ in \eqref{ssolve} and \eqref{wsolve} (which are necessarily unique). Hence, it remains to deal with $\hat{i}$ and $\hat{i}_E$. We treat the case of positive initial condition  in Section \ref{sec:eq>0}, and the case of zero initial condition in Section \ref{sec:eq=0}.

We now show that, the $t_*$ defined in \eqref{deft_*} is strictly smaller than $\hat{s}(0)$. To see this, we use an argument by contradiction. Assume that $t_*=\hat{s}(0)$. Then $\hat{i}_E(t)>0$ for all $0<t<  \hat{s}(0)$. Equations  \eqref{hati_E}, \eqref{ssolve} and \eqref{wsolve} together imply the existence of an $\ep>0$ so that
\begin{equation*}
	\frac{\mathrm{d}\hat{i}_E(t)}{\mathrm{d}t}\leq -1-\frac{\hat{i}_E(t)}{\hat{s}(t)}
\end{equation*} 
holds true for all $t\in (\hat{s}(0)-\ep, \hat{s}(0))$. It follows that
\begin{equation}\label{ddti/e}
	\frac{\mathrm{d}}{\mathrm{d}t}\left(\frac{\hat{i}_E(t)}{\hat{s}(t)}\right)=\frac{1}{\hat{s}(t)}\left(\frac{\mathrm{d}\hat{i}_E(t)}{\mathrm{d}t}+\frac{\hat{i}_E(t)}{\hat{s}(t)}\right)    \leq -\frac{1}{\hat{s}(t)}.
\end{equation}
Integrating \eqref{ddti/e} from $\hat{s}(0)-\epsilon$ to $t$, we see that
\begin{equation*}
	\frac{\hat{i}_E(t)}{\hat{s}(t)} -\frac{\hat{i}_E(\hat{s}(0)-\epsilon)}{\hat{s}(\hat{s}(0)-\epsilon)}
	\leq \int^t_{\hat{s}(0)-\epsilon} \frac{-1}{\hat{s}(u)}\mathrm{d}u=\log \left(\frac{\hat{s}(0)-t}{\ep}\right),
\end{equation*}
which goes to $-\infty$ as $t\to \hat{s}(0)$. We have thus arrived at a contradiction. Therefore, we must have $t_*<\hat{s}(0)$.

\subsection{Existence and uniqueness for positive initial condition}\label{sec:eq>0}

In this case the existence and uniqueness for the pair $(\hat{i},\hat{i}_E)$ follows from standard results in ODE theory. Indeed, the right-hand sides of \eqref{hati} and \eqref{hati_E} are Lipschitz functions of $\hat{i}$ and $\hat{i}_E$ when $\hat{i}_E$ is bounded away from 0. Existence theory for ODE implies that, given any solution $(\hat{i},\hat{i}_E)$  defined for small $t>0$, one can extend its domain to $t_*$ where $\lim_{t\to t_*}\hat{i}_{E}(t)=0$. We  set $\hat{i}_E(t_*)=0$.

As for the uniqueness, assume that there are two solution $(\hat{i}_1,\hat{i}_{E,1})$ and $(\hat{i}_2,\hat{i}_{E,2})$. Given any $\delta>0$, 
\begin{equation*}
	\begin{split}
		b_{\delta,1}&=\inf\{t\geq 0: \hat{i}_{E,1}\leq \delta\},\\
		b_{\delta,2}&=\inf\{t\geq 0: \hat{i}_{E,2}\leq \delta\},\\
		b_{\delta}&=b_{\delta,1}\wedge b_{\delta,2}.
	\end{split}
\end{equation*}
We define $t_{*,1}$ and $t_{*,2}$ to be the first time after 0 that $\hat{i}_{E,1}$ and $\hat{i}_{E,2}$ hit 0, respectively. Then $b_{\delta,1}\leq t_{*,1}<\hat{s}(0)$. Hence, $b_{\delta}<\hat{s}(0)$. By \eqref{hati} and \eqref{hati_E}, for $t\leq b_{\delta}$,
\begin{equation}\label{eq:ie1-ie2}
	\begin{split}
		\abs{\hat{i}_{E,1}(t)-\hat{i}_{E,2}(t)}&\leq 
		\int_0^t \frac{\abs{\hat{i}_{E,1}(u)-\hat{i}_{E,2}(u) }}{\hat{s}(u)}\mathrm{d}u+\frac{\omega \alpha }{\lambda }\int_0^t \abs{\hat{i}_1(u)-\hat{i}_2(u)}\mathrm{d}u\\
		&\leq  \left(\frac{1}{\hat{s}(0)-b_{\delta}} +\frac{\omega \alpha }{\lambda } \right) 
		\int_0^t \left(\abs{\hat{i}_{E,1}(u)-\hat{i}_{E,2}(u)}+
		\abs{\hat{i}_{1}(u)-\hat{i}_{2}(u)}
		\right)\mathrm{d}u,
	\end{split}
\end{equation}
and
\begin{equation}\label{eq:i1-i2}
	\begin{split}
		\abs{\hat{i}_1(t)-\hat{i}_2(t)}&\leq
		\frac{\gamma}{\lambda}   \int_0^t \abs{\frac{\hat{i}_1(u)}{\hat{i}_{E,1}(u)} -	\frac{\hat{i}_2(u)}{\hat{i}_{E,2}(u)}} \mathrm{d}u \\
		&\leq  \frac{\gamma}{\lambda}\     \int_0^t \frac{\abs{\hat{i}_{E,2}\hat{i}_1(u)-\hat{i}_2(u)\hat{i}_{E,1}(u) } }{\hat{i}_{E,1}(u)\hat{i}_{E,2}(u)}\mathrm{d}u\\
		&\leq  \frac{\gamma}{\lambda}\     \int_0^t \left( \frac{\abs{\hat{i}_1(u)-\hat{i}_2(u)}}{\hat{i}_{E,1}(u)} + \frac{\hat{i}_2(u)\abs{\hat{i}_{E,1}(u)-\hat{i}_{E,2}(u)  } }{\hat{i}_{E,1}(u)\hat{i}_{E,2}(u)} \right) \mathrm{d}u\\
		&\leq \frac{\gamma}{\lambda}\left(\frac{1}{\delta}+\frac{1}{\delta^2}\right) \int_0^t \left(\abs{\hat{i}_{E,1}(u)-\hat{i}_{E,2}(u)}+ \abs{\hat{i}_{1}(u)-\hat{i}_{2}(u)} \right)\mathrm{d}u.
	\end{split}
\end{equation}
Combining \eqref{eq:ie1-ie2} and \eqref{eq:i1-i2}, we see that for all $t\leq b_{\delta}$ and some constant $C>0$,
\begin{equation*}
  	\abs{\hat{i}_{E,1}(t)-\hat{i}_{E,2}(t)}+    \abs{\hat{i}_1(t)-\hat{i}_2(t)}\leq  C\int_0^t  \left( \abs{\hat{i}_{E,1}(u)-\hat{i}_{E,2}(u)}+ \abs{\hat{i}_{1}(u)-\hat{i}_{2}(u)}\right) \mathrm{d}u.
\end{equation*}
By Gronwall's inequality we deduce that, for all $t\leq b_{\delta}$ 
\begin{equation*}
	\hat{i}_{E,1}(t)=\hat{i}_{E,2}(t) \mbox{ and } \hat{i}_1(t)=\hat{i}_2(t).
\end{equation*}
Since $\delta$ can be taken to be arbitrarily small, we conclude that $t_{*,1}=t_{*,2}$, and for all $t\leq t_{*,1}$, 
\begin{equation*}
	\hat{i}_{E,1}(t)=\hat{i}_{E,2}(t) \mbox{ and } \hat{i}_1(t)=\hat{i}_2(t).
\end{equation*} 
This proves the uniqueness part. 

\subsection{Existence and uniqueness for zero initial condition} \label{sec:eq=0}

We start from the comparison principle for the system \eqref{hati}-\eqref{hati_E} under positive initial condition. 

\begin{lemma}\label{compare-1}
	Suppose that there are  two solutions $(\hat{i}_1,\hat{i}_{E,1})$ and $(\hat{i}_2,\hat{i}_{E,2})$ such that
	\begin{equation*}
		\hat{i}_1(0)\leq \hat{i}_2(0),\  \hat{i}_{E,1}(0)\leq  \hat{i}_{E,2}(0)  \mbox{ and } \hat{i}_{E,1}(0)> 0.
	\end{equation*} 
	Then for all $t\leq t_{*,2}$,
	\begin{equation*}
		\hat{i}_1(t)\leq \hat{i}_2(t) \mbox{  and  }  \hat{i}_{E,1}(t)\leq  \hat{i}_{E,2}(t). 
	\end{equation*}
\end{lemma}

\begin{proof}
	Denote the right-hand side of \eqref{hati} and \eqref{hati_E}
	by $F_1(\hat{i},\hat{i}_E,t)$ and $F_2(\hat{i},\hat{i}_E,t)$, 
	respectively. In other words, $F_1$ and $F_2$ are defined by
	\begin{equation*}
    	F_1(x,y,t)=-\frac{\gamma}{\lambda}\frac{x}{y}+1,
  	\end{equation*}
	and
	\begin{equation*}
    	F_2(x,y,t)=-1-\frac{\gamma}{\lambda}+\mu (\hat{s}(0)-t)-\frac{y}{\hat{s}(0)-t}+2\frac{\omega \alpha }{\lambda}(\hat{s}(0)-t)\log \frac{\hat{s}(0)}{\hat{s}(0)-t}-\frac{\omega}{\lambda}(1-\alpha+\alpha (1-x)).
  	\end{equation*} 
	Here we have substituted the explicit expressions of \eqref{ssolve} and \eqref{wsolve} into the right-hand side of \eqref{hati_E}. 
	It is clear that for any $x_1,x_2,t\geq 0$ and $y_1,y_2>0$, we have 
	\begin{equation*}
    	F_1(x_1,y_1,t)\leq F_1(x_2,y_2,t) \quad \mbox{ if }x_1=x_2 \mbox{ and } y_1\leq y_2,
  	\end{equation*}
	and 
	\begin{equation*}
    	F_2(x_1,y_1,t)\leq F_2(x_2,y_2,t)  \quad \mbox{ if }y_1=y_2 \mbox{ and } x_1\leq x_2.
  	\end{equation*}
	Thus, the system \eqref{hati}-\eqref{hati_E} satisfies the Kamker-M\"uller conditions, which in turn implies that this system is monotone with respect to its initial condition. See \cite[Theorem 3.2 and  Page 285]{MR1319817}.
\end{proof}

\begin{remark}\label{rem:backcom}
	Lemma \ref{compare-1} implies the following comparison principle. If 
	\begin{equation*}
    	\inf_{0\leq u\leq t} \min\{\hat{i}_{E,1}(u),\hat{i}_{E,2}(u)\}>0, \ \hat{i}_1(t)\leq \hat{i}_2(t), \ \hat{i}_{E,1}(t)\leq \hat{i}_{E,2}(t), 
  	\end{equation*}
	and 
	\begin{equation*}
    	\hat{i}_1(0)=\hat{i}_2(0) \quad (\mbox{resp. }  \hat{i}_{E,1}(0)=\hat{i}_{E,2}(0)),
  	\end{equation*}
	then we have
	\begin{equation*}
    	\hat{i}_{E,1}(0)<\hat{i}_{E,2}(0) \quad (\mbox{resp. } \hat{i}_{1}(0)<\hat{i}_{2}(0)).
  	\end{equation*}
\end{remark}

\begin{remark}\label{rem:onecom}
	Suppose that there are two solutions $(\hat{i}_1,\hat{i}_{E,1})$ and $(\hat{i}_2,\hat{i}_{E,2})$ such that $\hat{i}_{E,1}(0)\leq \hat{i}_{E,2}(0)$ and $\hat{i}_1(t)\leq \hat{i}_2(t)$ for all $t\leq t_{*,1}\wedge t_{*,2}$. Then one can show that $\hat{i}_{E,1}(t)\leq \hat{i}_{E,2}(t)$ for all $t\leq t_{*,1}\wedge t_{*,2}$.
\end{remark}

\subsubsection{Proof of the existence}

Recall that for zero  initial condition $\hat{i}(0)=\hat{i}_E(0)=0$ and $\hat{s}(0)=1$. Using Lemma \ref{compare-1} we can deduce the existence of a  solution to \eqref{hati} and \eqref{hati_E}. Indeed, consider the pair of functions $(\hat{i}_{\ep},\hat{i}_{E,\ep})$ given by the solution to \eqref{hati}-\eqref{hati_E} with the initial condition $\hat{i}_{\ep}(0)=\hat{i}_{E,\ep}(0)=\ep$. (The existence of $(\hat{i}_{\ep},\hat{i}_{E,\ep})$ is guaranteed by the results proved in Section \ref{sec:eq>0}.) Set $t^{\ep}_{*}:=\inf\{t>0: \hat{i}_{E,\ep}(t)=0\}$.
We claim that
\begin{equation}\label{t_*lb}
	\liminf_{\ep \to 0}t^{\ep}_{*}>0.
\end{equation}

\begin{proof}[Proof of \eqref{t_*lb}]
	Using \eqref{hati_E}, we see that for some $C_{\ref{hati'}}>0$,
	\begin{equation}\label{hati'}
		\abs{ \hat{i}'_{E,\ep}}\leq C_{\ref{hati'}}, \quad \forall \ \ep\leq 1,\  t\leq 1/2. 
	\end{equation}
	Hence, $\hat{i}_{E,\ep}\leq \ep+C_{\ref{hati'}}t$. Consequently,
	\begin{equation}\label{hati'2}
		\begin{split}
			\frac{\mathrm{d}\hat{i}_{E,\ep}}{\mathrm{d}t}&=-1-\frac{\gamma}{\lambda}-\frac{\omega }{\lambda}+\mu(1-t)-\frac{\hat{i}_{E,\ep}}{1-t}+\frac{2\hat{w}}{\hat{s}}
			+\frac{\omega  \alpha}{\lambda}\hat{i}_{\ep}\\
			& \geq \mu-1-\frac{\gamma+\omega}{\lambda}-\mu t-\frac{\ep+C_{\ref{hati'}}t}{1-t}.
		\end{split}
	\end{equation}
	Thanks to the assumption $\lambda<\lambda_c=(\gamma+\omega)/(\mu-1 )$ for zero initial condition,
	\begin{equation}\label{hati'3}
		\mu-1-\frac{\gamma+\omega}{\lambda} > \mu-1-\frac{\gamma+\omega}{\lambda_c}= \mu-1-\frac{\gamma+\omega}{(\omega+\gamma)/(\mu-1)}=0.
	\end{equation}
	By \eqref{hati'}, \eqref{hati'2} and \eqref{hati'3}, 
  	\begin{equation*} 
	 	\hat{i}'_{E,\ep}\geq 0 \quad \mbox{  for  } t\leq \left(\mu-1-\frac{\gamma+\omega}{\lambda}-2\ep \right) \min\left\{ \frac{1}{2\mu}, \left( \left(\mu-1-\frac{\gamma+\omega}{\lambda} \right)+2C_{\ref{hati'}} \right)^{-1}\right\}.
  	\end{equation*}
	This implies that $t^{\ep}_{*}$ is uniformly bounded from below for all small $\ep$.  Hence, \eqref{t_*lb} follows. 
\end{proof}

Define $t^0_{*}=\liminf_{\ep\to 0}t^{\ep}_{*}>0$. By Lemma \ref{compare-1}, for each $t< t^0_{*}$, $\hat{i}_{\ep}(t)$ and $\hat{i}_{E,\ep}(t)$  are both decreasing in $\ep$. Therefore, $t^{\ep}_{*}$  actually monotonically decreases to $t^0_{*}$, and we can define
\begin{equation*}
	\begin{split}
		\hat{i}(t)&=\lim_{\ep\to 0}\hat{i}_{\ep}(t),\\
		\hat{i}_E(t)&=\lim_{\ep\to 0}\hat{i}_{E,\ep}(t).
	\end{split}
\end{equation*}
Clearly $(\hat{i}(t),\hat{i}_E(t))$ satisfies \eqref{hati} and \eqref{hati_E} with $\hat{i}(0)=\hat{i}_E(0)=0$. Thus, we obtain a nontrivial solution of \eqref{hati}-\eqref{hati_E} since $t^0_*>0$. To prepare for the proof of uniqueness we need one more ingredient. Suppose that $(\hat{i}(t),\hat{i}_E(t))$ is a solution to \eqref{hati}-\eqref{hati_E}. We claim that
\begin{equation}\label{limsupiie}
	\limsup_{t\to 0}\frac{\hat{i}(t)}{\hat{i}_E(t)}<\infty.    
\end{equation}
To see this, note that by \eqref{hati}, 
\begin{equation}\label{iub}
	\hat{i}(t)\leq t, \quad \forall t\geq 0.
\end{equation}
Repeating the calculation in \eqref{hati'2}, 
\begin{equation}\label{limie/t}
	\hat{i}'_E(0)=\lim_{t\to 0}\frac{\hat{i}_E(t)}{t}=-1-\frac{\gamma}{\lambda}+\mu-\frac{\omega}{\lambda}>-\frac{\gamma+\omega}{\lambda_c}+\mu-1=-\frac{\gamma+\omega}{(\mu-1)/(\gamma+\omega)}+\mu-1=0. 
\end{equation}
Equation \eqref{limsupiie} follows from \eqref{iub} and \eqref{limie/t}.

\subsubsection{Proof of the uniqueness}

We now prove the uniqueness part.  Suppose that  $(\hat{i}_1(t),\hat{i}_{E,1}(t))$ and $(\hat{i}_2(t),\hat{i}_{E,2}(t))$ are two solutions that satisfy \eqref{hati}-\eqref{hati_E}.  We divide the proof into two steps. 

\textbf{Step 1}: We first show that either
\begin{equation}\label{compare0}
	\hat{i}_1(t)\leq  \hat{i}_2(t),\  \hat{i}_{E,1}(t)\leq  \hat{i}_{E,2}(t),\quad  \forall \ t\leq t_{*,1} \wedge t_{*,2},
\end{equation}
or \begin{equation}\label{compare00}
	\hat{i}_1(t)\geq  \hat{i}_2(t),\  \hat{i}_{E,1}(t)\geq  \hat{i}_{E,2}(t),\quad  \forall \ t\leq t_{*,1} \wedge t_{*,2},
\end{equation}
To see this, we consider two cases.
\begin{itemize}
	\item Case \textcircled{1}: either
	\begin{equation*}
    	\hat{i}_1(t)\leq  \hat{i}_2(t),\quad  \forall \ t\in (0,t_{*,1}\wedge t_{*,2}),
  	\end{equation*}
	or 
	\begin{equation*}
    	\hat{i}_1(t)\geq  \hat{i}_2(t),\quad  \forall \ t\in (0,t_{*,1}\wedge t_{*,2}).
 	\end{equation*}
	Using Remark \ref{rem:onecom} we see that either \eqref{compare0}
	or \eqref{compare00} holds true. 
	
	\item Case \textcircled{2}: the opposite of Case \textcircled{1}.  Without loss of generality let us assume  that  for some $0<t_3< t_1<t_{*,1}\wedge t_{*,2}$, we have
	$\hat{i}_{1}(t_3)>\hat{i}_{2}(t_3)$ and $\hat{i}_{1}(t_1)<\hat{i}_{2}(t_1)$. There must exist $t_2\in (t_3,t_1)$ such that $\hat{i}_1(t_2)=\hat{i}_2(t_2)$. By Lemma \ref{compare-1} and Remark \ref{rem:backcom}, we must have $\hat{i}_{E,2}(t_2)>\hat{i}_{E,1}(t_2)$, which, in turn, implies that $\hat{i}_{E,2}(t_3)>\hat{i}_{E,1}(t_3)$. Now we define 
	\begin{equation*}
    	t_4=\sup\{0\leq t\leq t_3: \hat{i}_1(t)=\hat{i}_2(t) \mbox{ or } \hat{i}_{E,1}(t)= \hat{i}_{E,2}(t)\}.
  	\end{equation*}
	We then have $0\leq t_3<t_3$. Moreover, for all $t_4<t<t_3$, we have 
  	\begin{equation}\label{compare6}
    	\hat{i}_1(t)>\hat{i_2}(t) \mbox{ and } \hat{i}_{E,1}(t)< \hat{i}_{E,2}(t).
  	\end{equation}
  	Now there are two possible scenarios:
  	\begin{itemize}
	 	\item Scenario 1:  $\hat{i}_1(t_4)=\hat{i}_2(t_4)$ and $\hat{i}_{E,1}(t_4)\leq \hat{i}_{E,2}(t_4)$. Using this and equations  \eqref{hati} and \eqref{compare6},
    	\begin{equation}
      		\begin{split}
        		0 &= (\hat{i}_1(t_3)-\hat{i}_1(t_4))- (\hat{i}_2(t_3)-\hat{i}_2(t_4))\\
        		&=\frac{\gamma}{\lambda} \int_{t_4}^{t_3} \left(\frac{\hat{i}_2(u)}{\hat{i}_{E,2}(u)} -\frac{\hat{i}_1(u)}{\hat{i}_{E,1}(u)}\right) \mathrm{d}u<0,
      		\end{split}
	  	\end{equation}
    	which is a contradiction.  

	  	\item Scenario 2:   $\hat{i}_1(t_4)>\hat{i}_2(t_4)$ and $\hat{i}_{E,1}(t_4)= \hat{i}_{E,2}(t_4)$. Then we have $t_4\in (0,t_3)$.  Using Lemma \ref{compare-1}, we get $\hat{i}_1(t_2) \geq \hat{i}_2(t_2)$. But this contradicts with the assumption $\hat{i}_1(t_2) < \hat{i}_2(t_2)$ in Case \textcircled{2}.
	\end{itemize}
	Since both scenarios are impossible, we conclude that Case \textcircled{2}  cannot occur. 
\end{itemize}
We have now finished the proof of \eqref{compare0} and \eqref{compare00}.

\textbf{Step 2}: Without loss of generality we assume that \eqref{compare0} holds true.  Using \eqref{hati_E},
\begin{equation}\label{pfstep2}
	\begin{split}
		\hat{i}_{E,2}(t)-\hat{i}_{E,1}(t)&\leq 
		\int_0^t \frac{\omega \alpha }{\lambda } (\hat{i}_2(u)-\hat{i}_1(u))\mathrm{d}u.
	\end{split}
\end{equation}
Define the function 
\begin{equation*}
  	H(t):=\sup_{0\leq u\leq t}\abs{\hat{i}_{1}(u)-\hat{i}_2(u)}.
\end{equation*}
For $t$ small, using \eqref{hati}, \eqref{limie/t}, \eqref{limsupiie} and \eqref{pfstep2},  
\begin{equation*}
	\begin{split}
		\hat{i}_2(t)-\hat{i}_1(t)&= \frac{\gamma}{\lambda} \int_0^t \left( \frac{\hat{i}_1(u)}{\hat{i}_{E,1}(u)}-\frac{\hat{i}_2(u)}{\hat{i}_{E,2}(u)} \right) \mathrm{d}u\\
		&\leq \frac{\gamma}{\lambda}\int_0^t \left( \frac{\hat{i}_1(u)}{\hat{i}_{E,1}(u)}-\frac{\hat{i}_1(u)}{\hat{i}_{E,2}(u)} \right) \mathrm{d}u\\
		& =\frac{\gamma}{\lambda} \int_0^t  \frac{\hat{i}_{1}(u) (\hat{i}_{E,2}(u)-\hat{i}_{E,1}(u))}{\hat{i}_{E,1}(u)\hat{i}_{E,2}(u)}\mathrm{d}u\\
		&\leq C \int_0^t \frac{1}{u} \left(\int_0^u \left(\hat{i}_2(r)-\hat{i}_1(r)\right) dr\right) \mathrm{d}u.
	\end{split}
\end{equation*}
It follows that
\begin{equation*}
	H(t)\leq C \int_0^t \frac{1}{u} uH(u)\mathrm{d}u\leq C\int_0^t H(u)\mathrm{d}u. 
\end{equation*}
Since $H(0)=0$, applying the Gronwall's ineqaulity we deduce that $H(t)=0$ for all $t$ small. This implies that $\hat{i}_{E,1}(t)=\hat{i}_{E,2}(t)$ for all $t$ small enough (say, for $t\leq \ep_0$). From the proof of the results in Section \ref{sec:eq>0} and the fact that $\hat{i}_{E,1}(\ep_0)=\hat{i}_{E,2}(\ep_0)>0$, we see that $\hat{i}_{E,1}(t)=\hat{i}_{E,2}(t)$ for all $\ep_0\leq t\leq t_{*,1}\wedge t_{*,2}$. Hence, $t_{*,1}=t_{*,2}$ and $(\hat{i}_1(t),\hat{i}_{E,1}(t))=(\hat{i}_2(t),\hat{i}_{E,2}(t))$ for all $t\leq t_{*,1}$.

\section{Approximations of the SIR-$\omega$ process by differential equations}\label{sec:conv}

The proof of the convergences in Theorem \ref{limitode}  is based on the general strategy  of tightness  and uniqueness argument. The starting point is the  Dynkin's formula (see, e.g., Chapter 4, Proposition 1.7 in \cite{MR838085}), which says that, if $V(t)$ is a Markov chain, then for any function $f$ in the domain of $L$ (the Markov generator for $V(t)$), we have 
\begin{equation*}
  	f(V(t))=f(V(0))+\int_0^t Lf(V(s))\mathrm{d}s+M(t),
\end{equation*}
where $M(t)$ is a martingale. This  is also called Dynkin's decomposition. Suppose for a sequence of Markov chains $V^{(n)}$ and a function $f$, the martingale term $M^{(n)}(t)$ converges to $0$ as $n\to\infty$, and the drift part $Lf(V^{(n)}(t))$ can be approximated by $g(f(V^{n}(t)),t)$ for some function $g$. Then, under the tightness of $f(V^{(n)}(t))$, its possible limit point $u(t)$ satisfies
\begin{equation*}
  	u(t)=u(0)+\int_0^t g(u(s),s)\mathrm{d}s.
\end{equation*}
If this equation has a  unique solution for any initial condition, then  $f(V^{(n)}(t))$ convergence to $u(t)$ in probability.  We will apply the tightness/uniqueness framework to $\hat{S}^{(n)}, \hat{W}^{(n)}, \hat{I}^{(n)}, \hat{I}^{(n)}_E$. 

Section \ref{sec:conv} is organized  into three parts. We first prove \eqref{eq:hatsw} in Section \ref{subsec:pfsw}, which is composed of three subsections \ref{subsub4.1.1} - \ref{subsub4.1.3}. In Section \ref{subsub4.1.1}, we figure out the Dynkin's decomposition of $\hat{S}^{(n)}$ and find that the drift term is  deterministic. Thus, by   controlling its martingale part, we  prove the convergence of $\hat{S}^{(n)}/n$. In Section \ref{subsub4.1.2} we prove $\hat{W}^{(n)}$ converges to $\hat{w}$ for $t\leq \hat{s}(0)-\epsilon$. In this regime, it is relatively easy to control the drift part in the Dynkin's decomposition of $\hat{W}^{(n)}$ and thus deduce its tightness. In Section \ref{subsub4.1.3} we lift this restriction, by controlling the oscillations of $\hat{W}^{(n)}/n$ and $\hat{w}$ in the interval $t\in [\hat{s}(0)-\epsilon,\hat{s}(0)]$. 

Section \ref{sec:ie>0} deals with the convergences of $(\hat{I}^{(n)}/n,\hat{I}_E^{(n)}/n)$ under the condition $\hat{i}(0),\hat{i}_E(0)>0$. We write down the evolution equations for $\hat{I}^{(n)}$ and $\hat{I}_E^{(n)}$ in Section  \ref{subsub4.2.1}. In Section \ref{subsub4.2.2} we first state  the convergence of $(\hat{I}^{(n)}/n, \hat{I}^{(n)}_E/n)$ up to the time when $\hat{I}_E^{(n)}$ drops below $\delta n$ for some small $\delta>0$. Then we remove this restriction  by again estimating the fluctuations of $\hat{I}^{(n)}$ and $\hat{I}^{(n)}_E$ when $t$ is near  $t_*\wedge \hat{\tau}^{(n)}$. 
 
Finally,  in Section \ref{sec:ie=0} we prove \eqref{eq:hatx} for  zero initial condition case where $\hat{i}(0)=\hat{i}_E(0)=0$.

\subsection{Proof of \eqref{eq:hatsw}: convergence of $\hat{S}^{(n)}(t)/n$ and $\hat{W}^{(n)}(t)/n$}\label{subsec:pfsw}

\subsubsection{Proof of the convergence of $\hat{S}^{(n)}(t)/n$ to $\hat{s}(t)$}\label{subsub4.1.1}

We start from the analysis of $\hat{S}^{(n)}(t)$. Corresponding to the three types of Poisson processes, we can divide the change of $\hat{S}^{(n)}(t)$ into three categories.
\begin{itemize}
	\item Change of $\hat{S}^{(n)}(t)$ due to recovery Poisson processes is 0.
	\item Change of $\hat{S}^{(n)}(t)$ due to rewiring/dropping Poisson processes is 0.
	\item Change of $\hat{S}^{(n)}(t)$ due to infection Poisson Processes is  $-1$. (Hereafter, the minus sign in the change of certain quantity means  that quantity will decrease.) The rate for this to occur is 
	\begin{equation*}
    	\lambda \hat{I}^{(n)}_E(t) \times \frac{n}{\lambda \hat{I}^{(n)}_E(t)}=n.
  	\end{equation*}
	(The factor $\frac{n}{\lambda \hat{I}^{(n)}_E(t)}$ comes from the fact that we are considering the time-changed dynamics.)
\end{itemize}
Collecting these three types of changes and applying Dynkin's formula, we see that
\begin{equation}\label{hats2}
	\hat{S}^{(n)}(t)=\hat{S}^{(n)}(0)+\int_0^t (-n)\mathrm{d}u+M^{(n)}_1(t),
\end{equation}
where $M^{(n)}_1(t)$ is a martingale. Note that we can also define $M^{(n)}_1(t)$ for $t\geq \hat{\tau}^{(n)}$  ($\hat{\tau}^{(n)}$ is the first time that $\hat{I}^{(n)}_E$ reaches 0) by simply setting $M^{(n)}_1(t)=M_1(\hat{\tau}^{(n)})$. Then $M^{(n)}_1(t)$ is  a martingale defined for all $t\geq 0$. 

By \eqref{hats2}, the quadratic variation process of $M^{(n)}_1(t)$, $\langle M^{(n)}_1(t),M^{(n)}_1(t)\rangle$, is equal to the quadratic variation process of $\hat{S}^{(n)}(t)$. The quadratic variation process of $\hat{S}^{(n)}(t)$ can be bounded using  previous analysis for the change of $\hat{S}^{(n)}(t)$. In particular, since the clocks in infection Poisson processes can effectively ring at most $n$ times, we have 
\begin{equation}\label{quast}
	\langle \hat{S}^{(n)}(t),\hat{S}^{(n)}(t) \rangle =\sum_{t'\leq t} 	(\hat{S}^{(n)}(t')- \hat{S}^{(n)}(t'-))^2 \leq n \times 1^2=n,\quad \forall \ t\leq \hat{\tau}^{(n)}.
\end{equation}
Consequently, 
\begin{equation*}
	\sup_{t\geq 0} \E\left(M^{(n)}_1(t)^2\right) \leq n.
\end{equation*}
By $L^2$ maximal inequality applied to the martingale $M^{(n)}_1(t)$, we have 
\begin{equation}\label{m1t1}
	\E\left( \sup_{0\leq t\leq \hat{\tau}^{(n)} }M^{(n)}_1(t)^2\right)\leq 4n.
\end{equation}
Markov inequality then implies that 
\begin{equation}\label{m1t2}
	\P\left(\sup_{t \leq \hat{\tau}^{(n)}}\abs{M^{(n)}_1(t)}>n^{2/3} \right)\leq \frac{4n}{n^{4/3}}\leq 4n^{-1/3}.
\end{equation}
Dividing both sides of \eqref{hats2} by $n$, we get
\begin{equation}\label{hats3}
	\frac{ \hat{S}^{(n)}(t)}{n}=\frac{ \hat{S}^{(n)}(0)}{n}- t+\frac{M^{(n)}_1(t)}{n}.
\end{equation}
Recall that we have assumed that  $\hat{S}^{(n)}(0)/n \CP \hat{s}(0)$.  Combining \eqref{m1t2} and \eqref{hats3},
\begin{equation}\label{slim2}
	\sup_{t\leq \hat{\tau}^{(n)}}\abs{ \hat{S}^{(n)}(t)/n-(\hat{s}(0)-t)}\CP 0,
\end{equation}
as desired. 

\subsubsection{Proof of the convergence of $\hat{W}^{(n)}(t)/n$ to $\hat{w}(t)$ for $t$ away from $\hat{s}(0)$}\label{subsub4.1.2}

Now we analyze the evolution equation of  $\hat{W}^{(n)}(t)$.
\begin{itemize}
	\item Change of $\hat{W}^{(n)}(t)$ due to recovery Poisson processes is 0.
	\item Change of $\hat{W}^{(n)}(t)$ due to rewiring/dropping Poisson processes is  1 with probability $\alpha\hat{S}^{(n)}(t)/n$, and 0 with probability $1-\alpha\hat{S}^{(n)}(t)/n$. Thus, the mean of the change is $\alpha\hat{S}^{(n)}(t)/n$. The rate for such an event to occur is
	\begin{equation*}
    	\omega \hat{I}^{(n)}_E(t) \times \frac{n}{\lambda \hat{I}^{(n)}_E(t)}=\frac{n\omega}{\lambda}.
  	\end{equation*}
	\item Change of $\hat{W}^{(n)}(t)$ due to infection Poisson processes is
	\begin{equation*}
    	-\mbox{Binomial}(\hat{W}^{(n)}(t),2/\hat{S}^{(n)}(t)).
  	\end{equation*}
	The mean of the change is $-2\hat{W}^{(n)}(t)/\hat{S}^{(n)}(t)$.
	The rate for such an event to occur is 
	\begin{equation*}
    	\lambda \hat{I}^{(n)}_E(t) \times \frac{n}{\lambda \hat{I}^{(n)}_E(t)}=n.
  	\end{equation*}
\end{itemize}
Hence, we can write, by Dynkin's formula,  
\begin{equation}\label{hatweq}
	\hat{W}^{(n)}(t)=\hat{W}^{(n)}(0)+\int_0^t \left(\frac{ \alpha\hat{S}^{(n)}(u)}{n} \frac{n\omega}{\lambda} -n \frac{2\hat{W}^{(n)}(u)}{\hat{S}^{(n)}(u)}\right)\mathrm{d}u  +M^{(n)}_2(t). 
\end{equation}

We first control the martingale $M^{(n)}_2(t)$. Using our analysis for the change of $\hat{W}^{(n)}(t)$, 
\begin{equation*}
	\langle M^{(n)}_2(t),M^{(n)}_2(t) \rangle = \langle \hat{W}^{(n)}(t),\hat{W}^{(n)}(t) \rangle  \leq \sum_{i=1}^n N_i^2+ Q(t),
\end{equation*}
where $Q(t)$ is the number of arrivals of rewiring/dropping Poisson processes by time $t$ and $N_i$ is the number of edges added to vertex $i$ when vertex $i$ first becomes infected. Since $Q(t)$ has total  rate $\omega n/\lambda$, 
\begin{equation}\label{boundqt}
	\E(Q(t))\leq \frac{\omega n t}{\lambda}.
\end{equation}
By Lemma \ref{boundni} we see 
\begin{equation}\label{boundsumni}
	\E\left(\sum_{i=1}^n N_i^2 \right) \leq Cn.
\end{equation}
Combining \eqref{boundqt} and \eqref{boundsumni} we see that there exists some constant $C$ such that for all $t\leq 2$,
\begin{equation*}
	\E \langle M^{(n)}_2(t),M^{(n)}_2(t) \rangle  \leq Cn.
\end{equation*}
Similarly to \eqref{m1t2} we deduce that
\begin{equation}\label{m1t4}
	\P\left(\sup_{t \leq \hat{\tau}^{(n)} \wedge 2}\abs{M^{(n)}_2(t)}>n^{2/3} \right)\leq \frac{Cn}{n^{4/3}}\leq Cn^{-1/3}.
\end{equation}

To finish the proof of \eqref{eq:hatsw}, we need to show
\begin{equation}\label{strongwt}
	\sup_{t\leq \hat{s}(0)\wedge \hat{\tau}^{(n)}}\abs{\hat{W}^{(n)}(t)/n-\hat{w}(t)}\CP 0.
\end{equation}
We first prove a weaker version of \eqref{strongwt}. 
Fix any $\ep>0$. We now prove
\begin{equation}\label{weakwt}
	\sup_{0\leq t\leq (\hat{s}(0)-\ep) \wedge  \hat{\tau}^{(n)} }\abs{\hat{W}^{(n)}(t)/n-\hat{w}(t)}\CP 0.
\end{equation}

\begin{proof}[Proof of \eqref{weakwt}]
	We prove \eqref{weakwt} in two steps. First we  show that $\hat{W}^{(n)}(t)/n$ is a tight sequence in D[0,$\hat{s}(0)-\ep$], the space of C\`adl\`ag functions on the interval [0,$\hat{s}(0)-\ep$]. Then we show the uniqueness of possible sequential limit. 
	
	Since $\hat{W}^{(n)}(0)  \to \hat{w}(0)$, we know   $\{  \hat{W}^{(n)}(0)/n \}_{n\geq 1}$ is a tight sequence of random variables. To establish tightness of  $\{ \hat{W}^{(n)}(t)/n ,0\leq t\leq 1-\ep\}_{n\geq 1}$,  we need to show that for any fixed $\ep',\delta>0$, there exist $\theta>0$ and an integer $n_0$ so that for all $n\geq n_0$,
	\begin{equation}\label{tightness}
		\P\left (\sup_{\abs{t_1-t_2}\leq \theta, t_1<t_2\leq \hat{s}(0)-\ep  } \abs{ \frac{\hat{W}^{(n)}(t_1)}{n} - \frac{\hat{W}^{(n)} (t_2)} {n} }  \geq \delta \right)\leq \ep'.
	\end{equation}
	
	Assuming \eqref{tightness} for the moment, we see that  $\{\hat{W}^{(n)}(t)/n,0\leq t\leq \hat{s}(0)-\ep\}$, as an element of D[0,$\hat{s}(0)-\ep$],  satisfies condition (ii) of	Proposition 3.26 in \cite{MR1943877}. Consequently, $\{\hat{W}^{(n)}(t)/n,0\leq t\leq \hat{s}(0)-\ep\}_{n\geq 1}$ is a tight sequence. 
	
	Now we show that any sequential  limit of $\hat{W}^{(n)}(t)/n$ coincides with $\hat{w}(t)$. By the tightness of $\{\hat{W}^{(n)}(t)/n,t\geq 0\}_{n\geq 1}$, we see for any subsequence of $\hat{W}^{(n)}(t)/n$ we can extract a further subsequence that converges in distribution to a  process $\check{w}(t)$ with continuous sample path. By the Skorokhod representation theorem we can assume the convergence is actually in the almost sure sense, and we can also assume that $\hat{S}^{(n)}(t)/n$  converges a.s. to $\hat{s}(t)$ in the interval [0,$\hat{s}(0)-\ep$] and $M^{(n)}_2(t)/n$ converges to 0 a.s.	It remains to prove that the limit  point $\check{w}(t)$ coincides with $\hat{w}(t)$, which is then necessarily independent of the subsequence. But this is clearly the case. Indeed, by dividing both sides of \eqref{hatweq} by $n$ and then sending $n$ to $\infty$, we obtain that
	\begin{equation*}
		\check{w}(t)=\hat{w}(0)+\int_0^t \frac{\omega}{\lambda}\alpha \hat{s}(u)\mathrm{d}u-\int_0^t 2\frac{\check{w}(u)}{\hat{s}(u)}\mathrm{d}u,
	\end{equation*}
	for $t\leq \hat{s}(0)-\ep$. This equation has a unique solution $\check{w}(t)=\hat{w}(t)$. 
	
	It remains to prove \eqref{tightness}. We have
	\begin{equation}\label{wdiff}
		\frac{\hat{W}^{(n)}(t_2)}{n}-\frac{\hat{W}^{(n)}(t_1)}{n} =\int_{t_1}^{t_2}\left(\frac{\alpha \omega \hat{S}^{(n)}(u) }{\lambda n} -\frac{2\hat{W}^{(n)}(u)}{\hat{S}^{(n)}(u)}\right)\mathrm{d}u +\frac{M^{(n)}_2(t_2)-M^{(n)}_2(t_1)}{n}.
	\end{equation}
	We define the event
	\begin{equation}\label{s>ep}
		\Omega_{\ref{s>ep}}:=\{\hat{S}^{n}(t)\geq \ep n/2 ,\ \forall t\leq (\hat{s}(0)-\ep)\wedge \hat{\tau}^{(n)} \} \cap \{  \sup_{t \leq \hat{\tau}^{(n)} \wedge 2}\abs{M^{(n)}_2(t)}\leq n^{2/3}  \}. 
	\end{equation}
	By \eqref{slim2} and \eqref{m1t4},  $\P(\Omega_{\ref{s>ep}}) \to 1$ as $n\to\infty$. On $\Omega_{\ref{s>ep}}$, for $t_1<t_2<\hat{s}(0)-\ep$, using \eqref{wdiff}, 
	\begin{equation*}
		\begin{split}
			&\abs{\frac{\hat{W}^{(n)}(t_2)}{n}-\frac{\hat{W}^{(n)}(t_1)}{n}}\\
			\leq &(t_2-t_1)\left(\frac{\alpha \omega}{\lambda }+\frac{2\mu n}{\ep n/2}\right)+\frac{2}{n} \sup_{t \leq \hat{\tau}^{(n)} \wedge 1}\abs{M^{(n)}_2(t)}\\
			\leq & C\left((t_2-t_1)+n^{-1/3}\right).
		\end{split}
	\end{equation*}
	This proves \eqref{tightness} and thus also completes the proof of \eqref{weakwt}. 
\end{proof}

\subsubsection{Proof of the convergence of $\hat{W}^{(n)}(t)/n$ to $\hat{w}(t)$ for all $t\leq \hat{s}(0)\wedge \hat{\tau}^{(n)}$}\label{subsub4.1.3}

Having proved \eqref{weakwt}, we now use it to show the original statement \eqref{eq:hatsw}. In other words,  given any $\ep'>0$, we need to show
\begin{equation}\label{strongwt2}
	\lim_{n\to\infty} \P\left(  \sup_{0\leq t\leq \hat{s}(0)\wedge \hat{\tau}^{(n)} }\abs{\frac{\hat{W}^{(n)}(t)}{n}-\hat{w}(t)}>\ep' \right)=0.
\end{equation}
Denote the oscillation of a function $h(t)$ in any given interval $[a,b]$  by 
\begin{equation}\label{defosc}
	\mbox{Osc}\left[h(t), [a,b]\right]:=\sup_{a\leq u\leq v \leq b}\abs{h(u)-h(v)}.
\end{equation}
Observe that
\begin{equation}\label{wnt2}
	\begin{split}
		&\sup_{0\leq t\leq \hat{s}(0)\wedge \hat{\tau}^{(n)} }\abs{\frac{\hat{W}^{(n)}(t)}{n}-\hat{w}(t)} \\
		\leq & 	\sup_{0\leq t\leq (\hat{s}(0)-\ep) \wedge \hat{\tau}^{(n)} }\abs{\frac{\hat{W}^{(n)}(t)}{n}-\hat{w}(t)} + \mbox{Osc}\left[\frac{\hat{W}^{(n)}(t)}{n},[(\hat{s}(0)-\ep)\wedge \hat{\tau}^{(n)}, \hat{s}(0)\wedge \hat{\tau}^{(n)}]\right]\\
		& + \mbox{Osc}[\hat{w}(t),[\hat{s}(0)-\ep, \hat{s}(0)]]\\
		:= & I_1^{(n)}+I_2^{(n)}+I_3^{(n)}. 
	\end{split}
\end{equation}

Fix any $\ep'>0$, we need to find $\ep$  so that 
\begin{equation}\label{i1i2i3}
	\lim_{n\to\infty}\P(I_1^{(n)}>\ep'/3)=\lim_{n\to\infty}\P(I_2^{(n)}>\ep'/3)=0 \mbox{ and } I_3^{(n)}\leq \ep'/3. 
\end{equation}
Once \eqref{i1i2i3} is proved, \eqref{strongwt2} follows from the union bound.

Equation \eqref{weakwt} already implies that $\P(I_1^{(n)}>\ep'/3)\to 0$ for any fixed $\ep$ and $\ep'$, so we only need to estimate $I_2^{(n)}$ and $I_3^{(n)}$. From \eqref{wsolve}  we have the explicit expression for $\hat{w}(t)$:
\begin{equation*}
  	\hat{w}(t)=\frac{\omega \alpha }{\lambda}(\hat{s}(0)-t)^2 \log \frac{\hat{s}(0)}{\hat{s}(0)-t}.
\end{equation*}
We can compute its derivative
\begin{equation*}
	\hat{w}'(t)=\frac{\omega \alpha }{\lambda } \left(2(t-\hat{s}(0)) \log \frac{\hat{s}(0)}{\hat{s}(0)-t}+(\hat{s}(0)-t) \right).
\end{equation*}
From this we see that the derivative of $\hat{w}(t)$ is uniformly bounded.
\begin{equation*}
	\sup_{0\leq t\leq \hat{s}(0)} \abs{  \hat{w}'(t)}\leq C.
\end{equation*}
It follows that
\begin{equation*}
	I_3^{(n)}=\mbox{Osc}[\hat{w},[\hat{s}(0)-\ep, \hat{s}(0)]]\leq C\ep. 
\end{equation*}
Also note that 
\begin{equation*}
  	\hat{w}(\hat{s}(0)-\ep) = \frac{\omega \alpha}{\lambda }\ep^2 \log (\hat{s}(0)/\ep)\leq C \ep. 
\end{equation*}
Hence, by choosing $\ep$ small enough we can make 
\begin{equation}\label{i3bd}
	I_3^{(n)}\leq \ep'/3 \mbox{ and } \hat{w}(\hat{s}(0)-\ep)\leq \ep'/3.
\end{equation}
It remains to control $I_2^{(n)}$.  Note that 
\begin{align*}
  	& I_2^{(n)} \left\{\begin{array}{l}=0,\ \text{ if }\ \hat{\tau}^{(n)}<\hat{s}(0)-\ep;\\
  	\leq \sup_{\hat{s}(0)-\ep\leq t \leq  \hat{s}(0)\wedge \hat{\tau}^{(n)}} \hat{W}^{(n)}(t)/n, \ \text{ if }\ \hat{\tau}^{(n)}\geq \hat{s}(0)-\ep.\end{array}\right.
\end{align*} 
Define the event
\begin{equation}\label{ws0-ep}
	\Omega_{\ref{ws0-ep}}=\{\hat{\tau}^{(n)}<\hat{s}(0)-\ep\} \cup \{\hat{\tau}^{(n)}\geq \hat{s}(0)-\ep, \hat{W}(\hat{s}(0)-\ep)>\ep' n/12\}.
\end{equation}
By \eqref{weakwt}, 
\begin{equation*}
	\lim_{n\to\infty} \P\left( \Omega_{\ref{ws0-ep}} \right)=1. 
\end{equation*}
By \eqref{wdiff},  on $\Omega_{\ref{s>ep}}\cap \Omega_{\ref{ws0-ep}} \cap \{\hat{\tau}^{(n)}\geq \hat{s}(0)-\ep\}$, for all $\hat{s}(0)-\ep\leq t \leq \hat{s}(0)\wedge \hat{\tau}^{(n)}$, 
\begin{equation}\label{wdd}
	\begin{split}
		\frac{\hat{W}^{(n)}(t)}{n}&\leq \frac{\hat{W}^{(n)}(\hat{s}(0)-\ep)}{n}+\frac{\alpha \omega }{\lambda } \ep +\frac{2\sup_{t}\abs{M^{(n)}_2(t)}}{n} \\
		&\leq \frac{\ep'}{12}+\frac{\alpha \omega }{\lambda } \ep +\frac{2 n^{2/3} }{n},
	\end{split}
\end{equation}
which is smaller than $\ep'/3$ for $\ep$ small enough and $n$ large enough. 

Consequently, on $\Omega_{\ref{s>ep}}\cap \Omega_{\ref{ws0-ep}}$,
$I_2\leq \ep'/3$. It follows that
\begin{equation}\label{i2bd}
	\lim_{n\to\infty}\P(I_2^{(n)}>\ep'/3)\leq \lim_{n\to\infty}\P\left(\left(\Omega_{\ref{s>ep}}\cap \Omega_{\ref{ws0-ep}}\right)^c\right)=0,
\end{equation}
since both $\P(\Omega_{\ref{s>ep}})$ and $\P(\Omega_{\ref{ws0-ep}})$ tend to 1 as $n\to\infty$. Equation \eqref{i1i2i3} now follows from \eqref{i3bd} and \eqref{i2bd}. This completes the proof of \eqref{eq:hatsw}.

\subsection{Proof of \eqref{eq:hatx} for positive initial condition}\label{sec:ie>0}

\subsubsection{Analysis of the evolution equations for $\hat{I}^{(n)}$ and $\hat{I}^{(n)}_E$.}\label{subsub4.2.1}

We first analyze $\hat{I}^{(n)}(t)$.
\begin{itemize}
	\item Change of $\hat{I}^{(n)}(t)$ due to recovery Poisson processes is $-1$. The rate for such an event to occur is 
	\begin{equation*}
    	\gamma \hat{I}^{(n)}(t) \times \frac{n}{\lambda \hat{I}^{(n)}_E(t)}=\frac{ \gamma \hat{I}^{(n)}(t) n}{\lambda \hat{I}^{(n)}_E(t)}.
  	\end{equation*}
	\item Change of $\hat{I}^{(n)}(t)$ due to rewiring/dropping Poisson processes is 0.
	\item Change of  $\hat{I}^{(n)}(t)$ due to infection Poisson processes is 1. The rate for such an event to occur is 
	\begin{equation*}
    	\lambda \hat{I}^{(n)}_E(t) \times \frac{n}{\lambda \hat{I}^{(n)}_E(t)}=n.
  	\end{equation*}
\end{itemize}
By Dynkin's formula again, we can decompose $\hat{I}^{(n)}(t)$ as follows: 
\begin{equation}\label{hatieq}
	\hat{I}^{(n)}(t)=\hat{I}^{(n)}(0)+\int_0^t\left(-\frac{ \gamma \hat{I}^{(n)}(t) n}{\lambda \hat{I}^{(n)}_E(t)}+n \right)  \mathrm{d}u+M^{(n)}_3(t). 
\end{equation}
The quadratic variation of the martingale $M^{(n)}_3(t)$ can be bounded in the similar way to $M^{(n)}_1(t)$:
\begin{equation*}
	\langle M^{(n)}_3(t),M^{(n)}_3(t) \rangle=\langle \hat{I}^{(n)}(t),  \hat{I}^{(n)}(t) \rangle  \leq \sum_{i=1}^n 1^2+\sum_{i=1}^n 1^2=2n.
\end{equation*} 
Similarly to \eqref{m1t2},
\begin{equation}\label{m3bd}
	\P\left(\sup_{t \leq  \hat{\tau}^{(n)}}\abs{M^{(n)}_3(t)} > n^{2/3} \right) \leq \frac{8n}{n^{4/3}}\leq 8n^{-1/3}.
\end{equation}

Now we turn to $\hat{I}^{(n)}_E(t)$. Let $\hat{I}^{(n)}_E(i,t)$ be the infected edges of vertex $i$ at time $t$. If $i$ is not infected at time $t$ then $\hat{I}^{(n)}_E(i,t)=0$. 
We can classify the change of $\hat{I}^{(n)}_E(t)$ as follows.
\begin{itemize}
	\item Change of $\hat{I}^{(n)}_E(t)$ due to recovery Poisson processes is  $-\hat{I}^{(n)}_E(i,t)$ if vertex $i$ recovers. The total rate of change is equal to
	\begin{equation*}
    	-\frac{n}{\lambda \hat{I}^{(n)}_E(t) } \sum_{i=1}^n  \gamma \hat{I}^{(n)}_E(i,t) = -\frac{n\gamma}{\lambda}.
  	\end{equation*}
	\item Change of $\hat{I}^{(n)}_E(t)$ due to rewiring/dropping Poisson processes is  $-1$ with probability $1-\alpha + \alpha(1-\hat{I}^{(n)}(t)/n)$, and 0 with probability $\alpha \hat{I}^{(n)}(t)/n$. 
	Thus, the mean of the change is $-(1-\alpha + \alpha(1-\hat{I}^{(n)}(t)/n))$.    The rate for such an event to occur is
	\begin{equation*}
    	\omega \hat{I}^{(n)}_E(t) \times \frac{n}{\lambda \hat{I}^{(n)}_E(t)}=\frac{n\omega}{\lambda}.
  	\end{equation*}
	\item Change of  $\hat{I}^{(n)}_E(t)$ due to Infection Poisson processes is equal to
	\begin{equation*}
    	-1+\mbox{Poisson}\left(\frac{(\hat{S}^{(n)}(t)-1)\mu}{n}\right)+\mbox{Binomial}\left(\hat{W}^{(n)}(t),2/\hat{S}^{(n)}(t)\right)-\mbox{Binomial}\left(\hat{I}^{(n)}_E(t)-1, \frac{1}{\hat{S}^{(n)}(t)}\right).
  	\end{equation*}
	To see this, recall the description of the  Infection Poisson Processes in the third bullet of the construction of the SIR-$\omega$ model in Section \ref{s:construct}. The first term `-1' is because each infection will cost an infected edge to infect the susceptible vertex. The appearance of the second and third terms come from \eqref{infpoi}. The last term is due to that  each infected edge is deleted with probability $1/\hat{S}^{(n)}(t)$ in each infection event. 
	
	The mean of the change is 
	\begin{equation*}
    	-1+  \frac{(\hat{S}^{(n)}(t)-1)\mu}{n}+\frac{2\hat{W}^{(n)}(t)}{\hat{S}^{(n)}(t)} - \frac{\hat{I}^{(n)}_E(t)-1}{\hat{S}^{(n)}(t)}.
  	\end{equation*}
	The rate for such an event to occur is 
	\begin{equation*}
    	\lambda \hat{I}^{(n)}_E(t) \times \frac{n}{\lambda \hat{I}^{(n)}_E(t)}=n.
  	\end{equation*}
\end{itemize}

We now decompose $\hat{I}^{(n)}_E(t)$ into the drift part and martingale part. 
\begin{equation}\label{hatiEeq}
	\begin{split}
		\hat{I}^{(n)}_E(t)=&\hat{I}^{(n)}_E(0) +\int_0^t   \left( -\frac{n\gamma}{\lambda} \right)\mathrm{d}u + \int_0^t   \frac{n\omega}{\lambda}\left(  -\left(1-\alpha + \alpha\left(1-\frac{\hat{I}^{(n)}(u)}{n}\right) \right)  \right) \mathrm{d}u\\
		& +\int_0^t n\left( -1+  \frac{(\hat{S}^{(n)}(u)-1)\mu}{n}+\frac{2\hat{W}^{(n)}(u)}{\hat{S}^{(n)}(u)} - \frac{\hat{I}^{(n)}_E(u)-1}{\hat{S}^{(n)}(u)} \right)\mathrm{d}u+M^{(n)}_4(t). 
	\end{split}
\end{equation}
Similarly to the analysis for the martingales $M^{(n)}_1,M^{(n)}_2, M^{(n)}_3$, we can show that
\begin{equation}\label{m4bd}
	\P\left(\sup_{t \leq \hat{\tau}^{(n)}\wedge ((\hat{s}(0)+t_*)/2) }\abs{M^{(n)}_4(t)} > n^{2/3} \right) \leq C'n^{-1/3}.
\end{equation}
We deferred the detailed proof of \eqref{m4bd} to Appendix \ref{sec:ap1}.

\subsubsection{Proof of \eqref{eq:hatx} for $\hat{I}^{(n)}$ and $\hat{I}^{(n)}_E$}\label{subsub4.2.2}

For any $\delta>0$, we can define $\hat{\tau}^{(n)}_{\delta}$ by
\begin{equation*}
	\hat{\tau}^{(n)}_{\delta}=\inf\{t\geq 0: \hat{I}^{(n)}_E(t)\leq \delta n\}.
\end{equation*}
Let 
\begin{equation*}
	\hat{\mathbf{Y}}^{(n)}(t):=(\hat{I}^{(n)}(t),\hat{I}^{(n)}_E(t));
\end{equation*}
\begin{equation*}
	\hat{\mathbf{y}}(t):=(\hat{i}(t),\hat{i}_E(t)).
\end{equation*}
In order to prove \eqref{eq:hatx}, we first observe the following weaker version of it: 
\begin{equation}\label{eq:haty}
	\sup_{0\leq t\leq t_* \wedge  \hat{\tau}^{(n)}_{\delta}}\abs{\hat{\mathbf{Y}}^{(n)}(t)/n-\hat{\mathbf{y}}(t)}\CP 0.
\end{equation}
We defer the proof of  \eqref{eq:haty} to Appendix \ref{app:haty}. To prove  the original version \eqref{eq:hatx}, we need the following two lemmas. 

\begin{lemma}\label{lem:taudelta}
	Consider the case of positive initial condition. For any $\ep>0$, there exists a $\delta_* \in (0,\hat{i}_E(0))$, s.t. for  any $\delta<\delta_*$, we have 
	\begin{equation*}
		\lim_{n\to\infty}\P(\hat{\tau}^{(n)}_\delta\geq t_*-\ep)=1.
	\end{equation*} 
\end{lemma}

\begin{proof}[Proof of Lemma \ref{lem:taudelta}]
	Observe that for any $\ep>0$, we have 
	\begin{equation}\label{infie}
		a:= \inf_{0\leq t\leq t_*-\ep}\hat{i}_E(t)>0. 
	\end{equation}
	We now set
	\begin{equation*}
		\delta<\delta_*:=\min\left\{\frac{a}{2},\frac{\hat{i}_E(0)}{2} \right\}.
	\end{equation*}
	Note that, if $\hat{\tau}^{(n)}_\delta\leq t_*-\ep$, then 
	\begin{equation*}
    \hat{\tau}^{(n)}_\delta\wedge  t_*=\hat{\tau}^{(n)}_\delta.
  \end{equation*}
	Hence, by \eqref{eq:haty}, 
	\begin{equation}\label{a/2}
		\begin{split}
			\lim_{n\to\infty}\P\left(\hat{\tau}^{(n)}_\delta\leq t_*-\ep,\ \abs{\frac{\hat{I}^{(n)}_E(\hat{\tau}^{(n)}_\delta)}{n}-{\hat{i}_E(\hat{\tau}^{(n)}_\delta) }}\geq \frac{a}{2} \right)=0.
		\end{split}
	\end{equation} 
	On the other hand, we claim that  the event 
	\begin{equation}\label{taundelta}
		\Omega_{\ref{taundelta}}=    \left\{\hat{\tau}^{(n)}_\delta\leq t_*-\ep,\ \abs{\frac{\hat{I}^{(n)}_E(\hat{\tau}^{(n)}_\delta)}{n}-{\hat{i}_E(\hat{\tau}^{(n)}_{\delta} }}< \frac{a}{2} \right\}
	\end{equation}
	is empty. Indeed, by \eqref{infie} and the definition of $\hat{\tau}^{(n)}_{\delta}$ and $\delta_*$, on $\Omega_{\ref{taundelta}}$,
	\begin{equation*}
		\frac{\hat{I}^{(n)}_E(\hat{\tau}^{(n)}_\delta)}{n} \leq \delta \leq \frac{a}{2} \mbox{ and } \hat{i}_E(\hat{\tau}^{(n)}_{\delta} ) \geq a.
	\end{equation*}
	Consequently, 
	\begin{equation*}
    	\abs{\frac{\hat{I}^{(n)}_E(\hat{\tau}^{(n)}_\delta)}{n}-{\hat{i}_E(\hat{\tau}^{(n)}_\delta) }}\geq a-\frac{a}{2}= \frac{a}{2}.
  	\end{equation*}
	Hence,
	\begin{equation}\label{2a/2}
		\Omega_{\ref{taundelta}}=\emptyset. 
	\end{equation}
	Combining \eqref{a/2} and \eqref{2a/2},
	\begin{equation*}
		\begin{split}
			\lim_{n\to\infty}\P\left(\hat{\tau}^{(n)}_\delta\leq t_*-\ep\right)=
			\lim_{n\to\infty}\P\left(\hat{\tau}^{(n)}_\delta\leq t_*-\ep, \ \abs{\frac{\hat{I}^{(n)}_E(\hat{\tau}^{(n)}_\delta)}{n}-{\hat{i}_E(\hat{\tau}^{(n)}_\delta) }}\geq \frac{a}{2} \right)=0,
		\end{split}
	\end{equation*}
	as desired. 
\end{proof}

We will also need the following Lemma \ref{lem:i=0}, whose proof is deferred to Appendix \ref{sec:ap2}. 

\begin{lemma}\label{lem:i=0}
	We have the convergence
	\begin{equation*}
		\lim_{t\to t_*} \hat{i}(t)=0.  
	\end{equation*}
\end{lemma}

We now come back to the proof of \eqref{eq:hatx}. Given any $\ep'>0$, we need to prove
\begin{equation}\label{0yI}
	\lim_{n\to\infty}\P\left( \sup_{0\leq t\leq t_*\wedge \hat{\tau}^{(n)}}\abs{\frac{\hat{I}^{(n)}(t)}{n}-\hat{i}(t)}>\ep' \right)=0,
\end{equation}
and 
\begin{equation}\label{0yIE}
	\lim_{n\to\infty}\P\left( \sup_{0\leq t\leq t_*\wedge \hat{\tau}^{(n)}}\abs{\frac{\hat{I}^{(n)}_E(t)}{n}-\hat{i}_E(t)} >\ep'\right)=0.
\end{equation}

\begin{proof}[Proof of \eqref{0yI} and \eqref{0yIE}] 
  Given the $\ep'>0$, we let $\ep>0$ be some number to be determined. By Lemma \ref{lem:taudelta},  there exists a $\delta>0$ such that the event $\{\hat{\tau}^{(n)}_{\delta}\geq t_*-\ep\}$ holds with high probability. On this event, 
	\begin{equation*}
		[0,t_*\wedge \hat{\tau}^{(n)}] = [0,t_*\wedge\hat{\tau}^{(n)}_{\delta}] \cup [t_*-\epsilon,t_*\wedge\hat{\tau}^{(n)}]. 
	\end{equation*}
	Consequently,
	\begin{equation*}
		\begin{split}
			& \sup_{0\leq t\leq t_*\wedge \hat{\tau}^{(n)}}\abs{\frac{\hat{I}^{(n)}(t)}{n}-\hat{i}(t)} \\
			\leq & \sup_{0\leq t\leq t_*\wedge\hat{\tau}^{(n)}_{\delta}}\abs{\frac{\hat{I}^{(n)}(t)}{n}-\hat{i}(t)} + \mbox{Osc}\left[\frac{\hat{I}^{(n)}(t)}{n},[t_*-\ep, t_*\wedge\hat{\tau}^{(n)}]\right]\\ 
			&+\mbox{Osc}[\hat{i}(t),[t_*-\ep, t_*]]\\
			:= & J_1^{(n)}+J_2^{(n)}+J_3^{(n)}. 
		\end{split}
	\end{equation*}
	Here Osc stands for the oscillation of a function, as defined in \eqref{defosc}. 
	
	Our goal is to find an $\ep>0$ so that 
	\begin{equation}\label{yI}
		\lim_{n\to\infty}\P(J_1^{(n)}>\ep'/3)=\lim_{n\to\infty}\P(J_2^{(n)}>\ep'/3)=0 \mbox{ and } J_3^{(n)}\leq \ep'/3,
	\end{equation}
	Once \eqref{yI} is proved, \eqref{0yI} follows by the union bound. 
	
	Equation \eqref{ytightness} already implies that $\P(J_1^{(n)}>\ep'/3)\to 0$ for any fixed $\ep$ and, $\ep'$. By Lemma \ref{lem:i=0}, $J_3^{(n)}\leq \ep'/3$ provided that $\ep$ is small enough. Therefore,  to prove \eqref{yI}, it remains to estimate $J_2^{(n)}$.  For $t_1<t_2\leq t_*\wedge \hat{\tau}^{(n)}$,
	\begin{equation*}
		\begin{split}
			\frac{\hat{I}^{(n)}(t_2)}{n}-\frac{\hat{I}^{(n)}(t_1)}{n} \leq \int_{t_1}^{t_2} \left(1-\frac{\gamma\hat{I}^{(n)}(u)}{\lambda\hat{I}^{(n)}_E(u)} \right)\mathrm{d}u +\frac{2\sup_{t\leq \hat{\tau}^{(n)} \wedge 2} \abs{M^{(n)}_3(t)}}{n}.
		\end{split}
	\end{equation*}
	Thus, for any $t \in [t_*-\ep, t_*\wedge\hat{\tau}^{(n)}]$,
	\begin{equation}\label{1yI}
		\frac{\hat{I}^{(n)}(t)}{n} \leq \ep +2n^{-1}\sup_{t\leq \hat{\tau}^{(n)} \wedge t_*} \abs{M^{(n)}_3(t)}+\abs{\frac{\hat{I}^{(n)}(t_*-\ep)}{n}}.
	\end{equation}
	Repeating the derivation of \eqref{wdd}, one sees that the right-hand side of \eqref{1yI} is smaller than $\ep'/3$ with high probability for large $n$ and small $\ep$. This implies that $\lim_{n\to\infty}\P(J_2^{(n)}>\ep'/3)=0$ for $\ep$ small and completes the proof of \eqref{0yI}. Equation \eqref{0yIE} can be proved similarly.  
\end{proof}

Equation \eqref{eq:hatx} in the	case of positive initial condition now follows from \eqref{0yI} and \eqref{0yIE}. 

For future references we record the following corollary of Lemma \ref{lem:taudelta}. Its proof is deferred to Appendix \ref{sec:ap3}. 

\begin{corollary}\label{cor:lbtau}
	Suppose $\hat{i}(0)>0$ and $\hat{i}_E(0)>0$. For any $\ep>0$,
	\begin{equation}\label{eq:lbtau}
		\lim_{n\to\infty}\P(\hat{\tau}^{(n)}\geq t_*-\ep)=1.
	\end{equation}
	Consequently, for any $\ep>0$, we have 
	\begin{equation}\label{sizelb}
		\P\left(\Lambda^{(n)}\geq (1-\hat{s}(0)+t_*-\ep)n\right)=1.
	\end{equation}
\end{corollary}

\subsection{Proof of \eqref{eq:hatx} for zero initial condition}\label{sec:ie=0}

Conditionally on the event that a major outbreak occurs,  we may assume that for some $\ep_0>0$, $\Lambda^{(n)}>2\ep_0 n$. We claim that, with high probability, $\hat{\tau}^{(n)}\geq \ep_0$. Indeed, by \eqref{hats},
\begin{equation*}
	\begin{split}
		\lim_{n\to\infty} \P\left(\Lambda^{(n)}>2\ep_0 n, \hat{\tau}^{(n)}<\ep_0 \right) \leq &\lim_{n\to\infty}\P(\hat{S}^{(n)}(\hat{\tau}^{(n)})<(1-2\ep)n,  \hat{\tau}^{(n)}<\ep_0)\\
		\leq &\lim_{n\to\infty}\P\left(\abs{\hat{S}^{(n)}(\hat{\tau}^{(n)})-(1-\hat{\tau}^{(n)})n }\geq \ep n \right)=0.
	\end{split}
\end{equation*}

We need the following rough bounds for $\hat{I}^{(n)}(t)$ and $\hat{I}^{(n)}_E(t)$.

\begin{lemma}\label{rul}
	By lowering the value of $\ep_0$ if needed, we can find two constants $C_{\ref{rub}}$ and $C_{\ref{rlb}}$ such for any $\ep>0$,
	\begin{equation}\label{rub}
		\lim_{n\to\infty} \P\left(\hat{I}^{(n)}(t)+\hat{I}^{(n)}_E(t)\leq (C_{\ref{rub}}t+\ep)n,\  \forall t\leq \ep_0 \wedge \hat{\tau}^{(n)}\right)=1,
	\end{equation}
	and 
	\begin{equation}\label{rlb}
		\lim_{n\to\infty} \P\left(\hat{I}^{(n)}_E(t)\geq (C_{\ref{rlb}}t-\ep)n,\  \forall t\leq \ep_0 \wedge \hat{\tau}^{(n)}\right)=1.
	\end{equation}
\end{lemma}

\begin{proof}[Proof of Lemma \ref{rul}]
	Define the events
	\begin{equation}\label{rul1}
		\Omega_{\ref{rul1}}=\{ \sup_{0\leq t\leq \hat{\tau}^{(n)}\wedge t_*} \left(\abs{M^{(n)}_3(t)}+\abs{M^{(n)}_4(t)}\right)\leq \ep n/4 \},
	\end{equation}
	and 
	\begin{equation}\label{rul2}
		\Omega_{\ref{rul}}=\{\hat{S}^{(n)}(t)\geq (1-\ep-t)n, \ \forall t\leq 1\wedge \hat{\tau}^{(n)}\}.
	\end{equation}
	Using \eqref{m3bd}, \eqref{m4bd} and the assumption 
  	\begin{equation*}
    	\frac{\hat{I}^{(n)}(0)}{n}\to \hat{i}(0)=0, \quad \frac{\hat{I}^{(n)}_E(0)}{n}\to \hat{i}_E(0)=0,
  	\end{equation*}
	we see that
	\begin{equation*}
    	\lim_{n\to\infty} \P(\Omega_{\ref{rul1}}) = \lim_{n\to\infty} \P(\Omega_{\ref{rul2}})=1.
  	\end{equation*}
	Define 
	\begin{equation}\label{omeganew}
		\Omega_{\ref{omeganew}}=\Omega_{\ref{rul1}} \cap \Omega_{\ref{rul2}}.
	\end{equation}
	Then we have $\P(\Omega_{\ref{omeganew}})\to 1$. On $\Omega_{\ref{omeganew}}$, using \eqref{hatieq}, we get
	\begin{equation*}
    	\hat{I}^{(n)}(t) \leq \hat{I}^{(n)}(0)+t+M^{(n)}_3(t)\leq t+\ep n,
  	\end{equation*}
	and 
	\begin{equation*}
		\begin{split}
			\hat{I}^{(n)}_E(t)&\leq  \hat{I}^{(n)}_E(0)+n\int_0^t\left(
			\mu+\frac{2\hat{W}^{(n)}(u)}{\hat{S}^{(n)}(u)}
			\right) \mathrm{d}u+M^{(n)}_4(t)\\
			&\leq \ep n+n \int_0^t \left(\mu+\frac{4\mu }{1-u-\ep} \right)\mathrm{d}u\\
			& \leq \ep n+9\mu n t,
		\end{split}
	\end{equation*}
	provided that  $\ep$ and $t$ are both smaller than $1/4$.  This proves \eqref{rub}. To show \eqref{rlb}, note that
	\begin{equation*}
		\hat{I}^{(n)}_E(t)\geq n\int_0^t \left(  -\frac{\gamma}{\lambda}-\frac{\omega}{\lambda}-1+\mu \frac{ \hat{S}^{(n)}(u)-1}{n}-\frac{\hat{I}^{(n)}_E(u)-1}{\hat{S}^{(n)}(u)}\right)\mathrm{d}u+M^{(n)}_4(t).
	\end{equation*}
	On $\Omega_{\ref{omeganew}}$,  $\hat{S}^{(n)}(u)\geq (1-\ep-t)n$. In addition,  $\hat{I}^{(n)}_E(t)\leq (9\mu t+\ep)n$, as we have just shown. Hence, for sufficiently small $\ep$ and $t$,
	\begin{equation*}
		\begin{split}
			\hat{I}^{(n)}_E(t) &\geq n\left( \left(-\frac{\gamma+\omega+\lambda}{\lambda}+\mu-\mu \ep\right)t- \frac{\mu t^2}{2}-\int_0^t \frac{9\mu t+\ep}{1/2}\mathrm{d}u  \right)-\ep n\\
			& \geq n\left(-\frac{\gamma+\omega+\lambda}{\lambda}+\mu\right)t/2-2\ep n,
		\end{split}
	\end{equation*}
	which proves \eqref{rlb}. 
\end{proof}

Let $\ep_0$ be small enough so that Lemma \ref{rul} holds. Define 
\begin{equation*}
  	\hat{\tau}^{(n)}_{\delta,\ep_0}:=\inf\{t\geq \ep_0: \hat{I}^{(n)}_E(t)\leq \delta n\}.
\end{equation*}

We now prove the tightness of $\{(\hat{I}^{(n)}(t)/n,\hat{I}^{(n)}_E(t)/n),\  t\leq \hat{\tau}^{(n)}_{\delta,\ep_0}\wedge t_*\}$. 

\begin{lemma}\label{lem:deltaep0}
	For any fixed $\ep',\delta_1>0$, there exist  $\theta>0$ and  $n_0>0$ so that for $n\geq n_0$,
	\begin{equation}\label{itight}
		\P\left (\sup_{\abs{t_1-t_2}\leq \theta, t_1<t_2\leq \hat{\tau}^{(n)}_{\delta,\ep_0} \wedge t_*  } \abs{ \frac{\hat{I}^{(n)}(t_1)}{n} - \frac{\hat{I}^{(n)} (t_2)} {n} }  \geq \delta_1 \right)
		\leq \ep',
	\end{equation}
	and 
	\begin{equation}\label{ietight}
		\P\left (\sup_{\abs{t_1-t_2}\leq \theta, t_1<t_2\leq
			\hat{\tau}^{(n)}_{\delta,\ep_0} \wedge t_* } \abs{ \frac{\hat{I}^{(n)}_E(t_1)}{n} - \frac{\hat{I}^{(n)}_E (t_2)} {n} }  \geq \delta_1 \right)
		\leq \ep'.
	\end{equation}
\end{lemma}

\begin{proof}[Proof of Lemma \ref{lem:deltaep0}]
	To prove \eqref{itight}, we assume that $\delta_1<\min\{\ep_0, C_{\ref{rub}}\ep_0\}$ and take 
	\begin{equation}\label{para1}
		\ep_1=\theta_1=\frac{\delta_1}{4C_{\ref{rub}}}.
	\end{equation}
	We divide $[0,\hat{\tau}^{(n)}_{\delta,\ep_0}\wedge t_*]$ into two sub-intervals: $B_1=[0,\ep_1]$ and $B_2=[\ep_1.\hat{\tau}^{(n)}_{\delta,\ep_0}\wedge t_*]$. Observe that if $t_1\in B_1$ and $\abs{t_1-t_2}\leq \theta $, then both $t_1$ and $t_2$ lie in the interval $[0,\ep_0]$. For $\ep=\delta_1/4$, we have 
	\begin{equation}\label{rub00}
		C_{\ref{rub}}(\ep_1+\theta_1)+\ep=\frac{3\delta_1}{4}.
	\end{equation}
	Using \eqref{rub} (with $\ep=\delta_1/4$) and \eqref{rub00}, we see that, for some $n_1$ large enough and all $n\geq n_1$,
	\begin{equation}\label{itight3}
		\P\left (\sup_{0<t_2-t_1\leq \theta_1, t_1\in B_1} \abs{ \frac{\hat{I}^{(n)}(t_1)}{n} - \frac{\hat{I}^{(n)}(t_2)} {n} }  \geq \delta_1 \right)\leq \P\left(  \sup_{t\leq \ep_1+\theta_1} \hat{I}^{(n)}(t) \geq \frac{3\delta_1}{4} n \right)\leq \frac{\ep'}{2}. 
	\end{equation}
	On the other hand, if both $t_1$ and $t_2$ are in the interval $B_2$, then 
	\begin{equation}\label{itight2}
		\lim_{n\to\infty} \P\left( \hat{I}^{(n)}_E(t)\geq 
		\min\left\{\frac{C_{\ref{rlb}}\ep_1}{2},  \delta \right\}n,\  \forall t\in
		[\ep_1,\hat{\tau}^{(n)}_{\delta,\ep_0}\wedge t_*]| \hat{\tau}^{(n)}>\ep_0
		\right)=1.
	\end{equation}
	Indeed, by \eqref{rlb} with $\ep=C_{\ref{rlb}}\ep_1/2$,
	\begin{equation}\label{itight-2}
		\lim_{n\to\infty} \P\left(\hat{I}^{(n)}_E(t)\geq (C_{\ref{rlb}}\ep_1- C_{\ref{rlb}}\ep_1/2)n,\  \forall \ep_1 \leq t\leq  \ep_0 | \hat{\tau}^{(n)}>\ep_0\right)=1.
	\end{equation}
	Thus, \eqref{itight2} follows from \eqref{itight-2} and 
	the definition of $\hat{\tau}^{(n)}_{\delta,\ep_0}$.

	By repeating the proof of \eqref{ytightness}, one can show that \eqref{itight2} implies that the sequence $$\{(\hat{I}^{(n)}(t)/n,\hat{I}^{(n)}_E(t)/n),\ \ep_1\leq t\leq \hat{\tau}^{(n)}_{\delta,\ep_0}\wedge t_*\}_{n\geq 1}$$ is a tight. Consequently, we can find some  $\theta_2$ and $n_2>0$ such that  for all $n\geq n_2$,
	\begin{equation}\label{itight4}
		\P\left (\sup_{0<t_2-t_1\leq \theta_2, t_1,t_2\in B_2} \abs{ \frac{\hat{I}^{(n)}(t_1)}{n} - \frac{\hat{I} (t_2)} {n} }  \geq \delta_1 \right)
		\leq \frac{\ep'}{2}.
	\end{equation}
	Equation \eqref{itight} now follows from \eqref{itight3}, \eqref{itight4} and the union bound by setting $$\theta=\min\{\theta_1,\theta_2\} \mbox{  and  } n_0=\max\{n_1,n_2\}.$$
	Equation \eqref{ietight}
	can be proved in the same way. 
	Hence, we have shown the tightness of the sequence $\{(\hat{I}^{(n)}(t)/n,\hat{I}^{(n)}_E(t)/n),\ 0\leq t\leq \hat{\tau}^{(n)}_{\delta,\ep_0} \wedge t_*\}$. 
\end{proof}

With the tightness established, we can proceed in the same way as the proof of \eqref{eq:hatx} in Section \ref{sec:ie>0} to deduce the uniform convergence of $(\hat{I}^{(n)}(t)/n, \hat{I}^{(n)}_E(t)/n)$ to $(\hat{i}(t),\hat{i}_E(t))$ for $t\leq t_*\wedge \hat{\tau}$. We omit the details. This completes the proof of \eqref{eq:hatx} for the case of $\hat{i}(0)=\hat{i}_E(0)=0$.

\begin{remark}
	Using the same proofs one can show that Lemma \ref{lem:i=0} and Lemma \ref{cor:lbtau}  also hold for the case of zero initial condition (for Corollary \ref{cor:lbtau} we need to condition on a major outbreak). 
\end{remark}

\section{Final epidemic size of the SIR-$\omega$ model}\label{sec:finalsize}

In this section we prove Theorem \ref{finalsize}. We only give a detailed proof for   part (i), equation \eqref{1-st*}, since the proof for part (ii) is identical.  By Corollary \ref{cor:lbtau}, for any $\ep>0$, with high probability $\Lambda^{(n)}\geq (1-\hat{s}(0)+t_*-\ep)n$. In order to prove Theorem \ref{finalsize} it remains to prove the other direction:
\begin{equation}\label{ubtau}
	\lim_{n\to\infty}\P(\Lambda^{(n)}>(1-\hat{s}(0)+t_*+\ep)n)=0.
\end{equation}
Define the event
\begin{equation*}
  	\Omega_{\ref{taut000}}:=\{\hat{\tau}^{(n)}> t_*+\ep\}.
\end{equation*}
Using \eqref{hats} and equality $\Lambda^{(n)}=\hat{\Lambda}^{(n)}=n-\hat{S}^{(n)}(\hat{\tau}^{(n)})$, in order  to prove \eqref{ubtau} it suffices to prove
\begin{equation}\label{taut000}
	\lim_{n\to\infty}\P(\Omega_{\ref{taut000}})=0.
\end{equation}
Indeed, \eqref{ubtau} follows from \eqref{taut000}, \eqref{ssolve} and \eqref{eq:hatsw}. For the rest of this section we  fix a small $\ep>0$. Our goal is to prove \eqref{taut000}.  

Using  equation \eqref{hatiEeq} of $\hat{I}_E$, for $t\geq t_*$,
\begin{equation}\label{eqie}
	\begin{split}
		\hat{I}^{(n)}_E(t)=&\hat{I}^{(n)}_E(t_*)+
		\int_{t_*}^t   n\left(-1-\frac{\gamma}{\lambda}-\frac{\omega }{\lambda} +\mu \frac{\hat{S}^{(n)}(u)-1}{n}+\frac{2\hat{W}^{(n)}(u)}{\hat{S}^{(n)}(u)} \right) \mathrm{d}u\\
		&+\int_{t_*}^t n\left(\frac{\omega \alpha }{\lambda}\frac{\hat{I}^{(n)}(u)}{n}-\frac{\hat{I}^{(n)}_E(u)-1}{\hat{S}^{(n)}(u)}\right)\mathrm{d}u
		+M^{(n)}_4(t)-M_4^{(n)}(t_*).
	\end{split}
\end{equation}
We also have, by \eqref{hatieq}, for $t\geq t_*$, 
\begin{equation}\label{eqi}
	\hat{I}^{(n)}(t)=\hat{I}^{(n)}(t_*)+\int_{t_*}^t  n\left(-\frac{\gamma}{\lambda}\frac{\hat{I}^{(n)}(u)}{\hat{I}^{(n)}_E(u)}+1\right)\mathrm{d}u + M^{(n)}_3(t)-M_3^{(n)}(t_*).
\end{equation}
Define
\begin{equation}\label{deff2}
	\begin{split}
		F(t)&=-1-\frac{\gamma}{\lambda}-\frac{\omega}{\lambda}+\mu \hat{s}(t)+2\frac{\hat{w}(t)}{\hat{s}(t)},\\
		F_n(t)& = -1-\frac{\gamma}{\lambda}-\frac{\omega }{\lambda} +\mu \frac{\hat{S}^{(n)}(t)-1}{n}+\frac{2\hat{W}^{(n)}(t)}{\hat{S}^{(n)}(t)}. 
	\end{split}
\end{equation}
Clearly $F$ is a continuous function of $t$. Given any $\delta>0$, we can define an event 
\begin{equation*}
  	\Omega_{\ref{fn-f}}(\delta):=\{ \abs{ F_n(t)-F(t)}\leq \delta, \ \forall \ t\leq \hat{\tau}^{(n)}\wedge (t_*+\ep)\}.
\end{equation*}
Using \eqref{eq:hatsw}, for any $\delta>0$, 
\begin{equation}\label{fn-f}
	\lim_{n\to\infty} \P\left( \Omega_{\ref{fn-f}}(\delta) \right)=1.
\end{equation}
We define a random function $E^{(n)}(t)$ by setting $E(t)=0$ for $t\leq t_*$, and for $t\geq t_*$, 
\begin{equation}\label{defet}
	\begin{split}
		E^{(n)}(t)=  & \int_{t_*}^t n\left(F_n(u)-F(u)\right) \mathrm{d}u+    \hat{I}^{(n)}_E(t_*)+ \int_{t_*}^t n\left(\frac{\omega \alpha }{\lambda}\frac{\hat{I}^{(n)}(u)}{n}-\frac{\hat{I}^{(n)}_E(u)-1}{\hat{S}^{(n)}(u)}\right)\mathrm{d}u\\
		& +M^{(n)}_4(t)-M_4^{(n)}(t_*).
	\end{split}
\end{equation}
Then on the event $\Omega_{\ref{taut000}}$, for $t\geq t_*$, 
\begin{equation}\label{defEt}
	\hat{I}^{(n)}_E(t)=   \int_{t_*}^t F(u)\mathrm{d}u+E^{(n)}(t).
\end{equation}

We claim  that 
\begin{equation}\label{ft_*}
	F(t_*)\leq 0.
\end{equation}
To see this, recall that we have shown $\hat{i}(t_*)=0$ in Lemma \ref{lem:i=0}. We also have $\hat{i}_E(t_*)=0$ by the definition of $t_*$. Thus,
\begin{equation}\label{iet*1}
	\hat{i}'_E(t_*)=  -1-\frac{\gamma}{\lambda}+\mu \hat{s}(t_*)+2\frac{\hat{w}(t_*)}{\hat{s}(t_*)}-\frac{\omega}{\lambda}=F(t_*).
\end{equation}
On the other hand
\begin{equation}\label{iet*2}
	\hat{i}'_E(t_*)=\lim_{t\to t_*-} \frac{\hat{i}(t_*)-\hat{i}(t) }{t_*-t}\leq 0. 
\end{equation}
Equation \eqref{ft_*} thus follows from \eqref{iet*1} and \eqref{iet*2}. 

We now divide the proof of \eqref{taut000} into two subsections, according to $F(t_*)<0$ or $F(t_*)=0$. 

\subsection{Proof of \eqref{taut000} for $F(t_*)<0$}

By the continuity of $F$ around $t_*$, 
\begin{equation}\label{F<0}
	F(t)\leq -\frac{3F(t_*)}{4},\quad  \forall \ t\in [t_*,t_*+\ep],
\end{equation}
provided  $\ep$ is small. 
Define the event 
\begin{equation}\label{610}
	\Omega_{\ref{610}}= \{\hat{\tau}^{(n)}\leq t_*+\ep  \}
	\cup  \{\hat{\tau}^{(n)}> t_*+\ep ,\hat{I}^{(n)}(t)\leq \left(t-t_*-\frac{F(t_*)\lambda}{8\omega \alpha}\right)n, \ \forall  t_* \leq t\leq  \hat{\tau}^{(n)}\}.
\end{equation}
By \eqref{m3bd} and \eqref{hatieq},
\begin{equation}\label{omega610}
	\lim_{n\to\infty}\P(\Omega_{\ref{610}})=1.
\end{equation} 
On $\Omega_{\ref{610}}\cap \{ \hat{\tau}^{(n)}\geq t_*+\ep  \} $,  we have
\begin{equation}\label{ini}
	\int_{t_*}^{t_*+\ep}\hat{I}^{(n)}(t)\mathrm{d}t\leq \frac{\ep^2}{2} -\frac{F(t_*)\lambda}{8\omega \alpha} \ep.
\end{equation}
We set $\Omega_{\ref{611}}$ to be
\begin{equation}\label{611} 
  	\left\{\hat{\tau}^{(n)}\leq t_*+\ep \right\} \cup \left\{\hat{\tau}^{(n)}> t_*+\ep,\ \abs{\hat{I}^{(n)}_E(t_*)}\leq \frac{-F(t_*)\ep n}{8},\ \sup_{t\leq  \hat{\tau}^{(n)} \wedge  (t_*+\ep)} \abs{M^{(n)}_4(t)} \leq \frac{-F(t_*)\ep n}{8}\right\} .
\end{equation}
By \eqref{eq:hatx} (applied to $t=t_*$), the fact that $\hat{i}_E(t_*)=0$ and \eqref{m4bd} (assume $\ep<(\hat{s}(0)-t_*)/2$), 
\begin{equation}\label{omega611}
	\lim_{n\to\infty}\P(\Omega_{\ref{611}})=1.
\end{equation}

On the event $\Omega_{\ref{610}} \cap \Omega_{\ref{611}} \cap  \Omega_{\ref{fn-f}}(-F(t_*)/8)\cap \{\hat{\tau}^{(n)}> t_*+\ep\}$, by \eqref{F<0} and \eqref{ini},
\begin{equation*}
	\begin{split}
		\hat{I}^{(n)}_E(t+\ep)&\leq  (t_*+\ep-t_*)\frac{3F(t_*)n}{4}+E^{(n)}(t_*+\ep)\\
		&\leq \frac{3\ep F(t_*)n}{4} -\frac{F(t_*)\ep n}{8}-\frac{F(t_*)\ep n}{8} +\frac{n\omega \alpha}{\lambda} \left( \frac{\ep^2}{2} -\frac{F(t_*)\lambda}{8\omega \alpha} \ep \right)-\frac{\ep F(t_*)}{8}-\frac{\ep F(t_*)}{8}\\
		&=n\left(\frac{ F(t_*)}{8}\ep+\frac{\omega \alpha }{2\lambda}\ep^2 \right),
	\end{split}
\end{equation*}
which is smaller than $0$ if $\ep$ is small. 
Thus, on $\Omega_{\ref{610}} \cap \Omega_{\ref{611}} \cap  \Omega_{\ref{fn-f}}(-F(t_*)/8)\cap \{\hat{\tau}^{(n)}> t_*+\ep\}$, we necessarily have $\hat{\tau}^{(n)}\leq t_*+\ep$,  which is a contradiction. Consequently, 
\begin{equation}\label{610611}
	\Omega_{\ref{610}} \cap \Omega_{\ref{611}} \cap  \Omega_{\ref{fn-f}}(-F(t_*)/8)\cap \{\hat{\tau}^{(n)}> t_*+\ep\}=\emptyset.
\end{equation}
Combining \eqref{fn-f}, \eqref{610611}, \eqref{omega610} and \eqref{omega611}, we get
\begin{equation*}
	\lim_{n\to\infty}\P(\hat{\tau}^{(n)}> t_*+\ep)=0,
\end{equation*}
as desired. 

\subsection{Proof of \eqref{taut000} for $F(t_*)=0$}

By \eqref{deff2} and the assumption $\lambda>\lambda_c$,
\begin{equation}\label{f't1}
	F'(t)=-\mu+\frac{2\omega \alpha }{\lambda } \left(1 -\log \frac{\hat{s}(0)}{\hat{s}(t)}\right) \leq -\mu+ \frac{2\omega \alpha }{\lambda } <-\mu+\frac{2\omega \alpha }{\lambda_c}. 
\end{equation}
Recall that $\lambda_c=(\gamma+\omega)/(\mu-1)$. Thus, we have 
\begin{equation}\label{f't2}
	\begin{split}
		-\mu+\frac{2\omega \alpha }{\lambda_c}&= -\mu+\frac{2\omega \alpha  (\mu-1)}{\gamma+\omega}\\
		& =\left(\frac{2\omega \alpha }{\gamma+\omega}-1\right)\mu- \frac{2\omega \alpha}{\gamma+\omega}\\
		&=
		\left(\frac{2\omega \alpha }{\gamma+\omega}-1 \right)\left( \mu- \frac{2\omega \alpha }{2\omega \alpha -\omega -\gamma}\right).
	\end{split}
\end{equation}

If the condition \eqref{cond_cont} is satisfied, then   either 
\begin{equation*}
  	\frac{2\omega \alpha }{\gamma+\omega}-1 \leq \frac{2\omega \alpha }{\omega(2\alpha -1)+\omega}-1=0,
\end{equation*}
or 
\begin{equation*}
  	\frac{2\omega \alpha }{\gamma+\omega}-1>0 \mbox{ and }\mu\leq  \frac{2\omega \alpha }{2\omega \alpha -\omega -\gamma}.
\end{equation*}
By \eqref{f't2}, in both cases 
\begin{equation}\label{f't3}
	-\mu+\frac{2\omega \alpha }{\lambda_c}\leq 0.
\end{equation}
Therefore, we can set
\begin{equation}\label{f't4}
	C_{\ref{f't4}}:= -\left(    -\mu+\frac{2\omega \alpha }{\lambda}\right) >0.
\end{equation}
It follows from \eqref{f't1} that, for $t\geq t_*$, 
\begin{equation*}
  	F(t)\leq -C_{\ref{f't4}}(t-t_*).
\end{equation*}
Consequently, we have
\begin{equation}\label{intft}
	\int_{t_*}^t F(u)\mathrm{d}u\leq -\int_{t_*}^t C_{\ref{f't4}}(u-t_*)\mathrm{d}u=-\frac{C_{\ref{f't4}}}{2}(t-t_*)^2.
\end{equation}
Let $\ep_1$ be some number to be determined and define the events 
\begin{equation}\label{62-1}
	\Omega_{\ref{62-1}}=\left\{ \sup_{0\leq t\leq  \hat{\tau}^{(n)}} \abs{M^{(n)}_3(t)}\leq \ep_1 n\right\},
\end{equation}
By \eqref{m3bd},
\begin{equation}\label{omega620}
	\lim_{n\to\infty}\P(\Omega_{\ref{62-1}})=1.
\end{equation}
Consider the event  
\begin{align}\label{621}
	\Omega_{\ref{621}}=&
		\left\{\sup_{0\leq t\leq  \hat{\tau}^{(n)}\wedge (t_*+\ep) } \abs{M^{(n)}_4(t)} \leq \ep_1 n\right\} \cap 	\Omega_{\ref{fn-f}}(\ep_1) \cap \Omega_{\ref{62-1}}\\
		& \cap \left( \{\hat\tau\leq t_*+\ep\}  \cup \{ \hat\tau> t_*+\ep, \hat{I}^{(n)}_E(t_*)\leq \ep_1 n, \hat{I}^{(n)}(t)\leq (t-t_*+\ep_1)n,\ \forall t_*\leq t\leq \hat{\tau}^{(n)}  \}  \right)\nonumber.
\end{align}
By the fact that $\hat{i}_E(t_*)=0$, \eqref{hatieq}, \eqref{m3bd},
\eqref{m4bd} and \eqref{fn-f} (with the $\delta$ there set to be $\ep_1$),
\begin{equation*}
  	\lim_{n\to\infty}\P(\Omega_{\ref{621}})= 1.
\end{equation*}
On $\Omega_{\ref{621}}$, using the definition of $E^{(n)}(t)$ in \eqref{defet}, for $t_*\leq t \leq t_*+\ep$, 
\begin{equation}\label{etbd}
	\begin{split}
		E^{(n)}(t)\leq & \int_{t_*}^t (n\ep_1) \mathrm{d}u +\ep_1 n + \int_{t_*}^t n\frac{\omega \alpha }{\lambda}(u-t_*+\ep_1)\mathrm{d}u+2\ep_1 n\\
		\leq & \left(\frac{\omega \alpha }{2\lambda}(t-t_*)^2+ \left(1+\frac{\omega\alpha}{\lambda}\right)    \ep_1 (t-t_*) +3\ep_1\right) n.
	\end{split}
\end{equation}
If both $\ep$ and $\ep_1$ are small enough, then 
\begin{equation}\label{etbd2}
	\frac{\omega \alpha }{2\lambda}(t-t_*)^2+ \left(1+\frac{\omega\alpha}{\lambda}\right)    \ep_1 (t-t_*) \leq \frac{C_{\ref{f't4}}\gamma}{4\omega \alpha }(t-t_*),\quad  \forall \ t_*\leq t\leq t_*+\ep. 
\end{equation}
It follows from  \eqref{defEt}, \eqref{intft} \eqref{etbd} and \eqref{etbd2}  that, on $\Omega_{\ref{621}}$, for $t_*\leq t\leq t_*+\ep$,
\begin{equation}\label{iebd3}
	\hat{I}^{(n)}_E(t)\leq E^{(n)}(t)\leq \frac{C_{\ref{f't4}}\gamma}{4\omega \alpha }(t-t_*)n+3\ep_1 n. 
\end{equation}
Using \eqref{eqi} and \eqref{iebd3},  on $\Omega_{\ref{621}}$, 
\begin{equation*}
	\begin{split}
		0\leq    \hat{I}^{(n)}(t)&\leq \int_{t_*}^t  n\left(-\frac{\gamma}{\lambda}\frac{\hat{I}^{(n)}(u)}{\hat{I}^{(n)}_E(u)}+1\right)\mathrm{d}u+3\ep_1 n \\
		&\leq n\left(-\frac{\gamma}{\lambda}\int_{t_*}^t \frac{\hat{I}^{(n)}(u)}{ C_{\ref{f't4}}\gamma  (u-t_*)n/(4\omega \alpha)+3\ep_1 n}\mathrm{d}u  +t-t_*+3\ep_1 \right).
	\end{split}
\end{equation*}
Hence, we deduce that on $\Omega_{\ref{621}}$, for $\ep_1$ small enough, 
\begin{equation}\label{ibd4}
	\begin{split}
		\int_{t_*}^t \hat{I}^{(n)}(u)\mathrm{d}u &\leq \left( \frac{C_{\ref{f't4}}\gamma}{4\omega \alpha}  (t-t_*) +3\ep_1  \right)      \int_{t_*}^t \frac{\hat{I}^{(n)}(u)}{C_{\ref{f't4}}\gamma  (u-t_*)/(4\omega \alpha)+3\ep_1 }\mathrm{d}u\\
		&\leq n\left( \frac{C_{\ref{f't4}}\gamma}{4\omega \alpha}  (t-t_*) +3\ep_1  \right) \frac{\lambda}{\gamma} (t-t_*+3\ep_1)\\
		&\leq n\left(\frac{C_{\ref{f't4}}\lambda }{4\omega \alpha} (t-t_*)^2+\ep_1 \right). 
	\end{split}
\end{equation}
Apply \eqref{ibd4} to \eqref{defet}, on $\Omega_{\ref{621}}$,
\begin{equation*}
	\hat{I}^{(n)}_E(t)\leq \int_{t_*}^t F(u)\mathrm{d}u+  \frac{\omega \alpha }{\lambda}\left( n\frac{C_{\ref{f't4}}\lambda }{4\omega \alpha} (t-t_*)^2+\ep_1 n\right)+ \ep_1 n(t-t_*)+ 3\ep_1 n.
\end{equation*}
Using \eqref{intft}, 
\begin{equation}\label{iebd5}
	\hat{I}^{(n)}_E(t)\leq -\frac{C_{\ref{f't4}}}{4}(t-t_*)^2 +\left(\frac{\omega \alpha}{\lambda}+4\right)\ep_1 n.
\end{equation}
Using \eqref{iebd5}, by choosing $\ep_1$ small enough so that 
\begin{equation*}
  	\ep_1 <\frac{C_{\ref{f't4}}}{4} \left(\frac{\omega \alpha}{\lambda}+4\right)^{-1} \ep^2,
\end{equation*}
we can achieve $\hat{I}^{(n)}_E(t_*+\ep)\leq 0$ on $\Omega_{\ref{621}}$, which then implies that $\hat{\tau}^{(n)}<t_*+\ep$ on this event. Since $\P(\Omega_{\ref{621}})\to 1$ as $n\to\infty$, \eqref{taut000} follows. Thus, we have completed the proof of  \eqref{taut000} in Case \textcircled{2} ($F(t_*)=0$) and also the proof of part (i) of Theorem \ref{finalsize}.  As we mentioned in the beginning of this section, part (ii) can be proved in the same way. We omit the details. 

\section{Continuity of Phase transitions of the SIR-$\omega$ model}\label{sec:contphase}

In this section, we shall prove Theorem \ref{phase}. By Theorem \ref{finalsize}, it suffices to prove that $t_*\to 0$ as $\lambda \searrow\lambda_c=(\gamma+\omega)/(\mu-1)$. Recall  the definition of the function $F(t)$ in \eqref{deff2}. Using the expression of $F$  we can write
\begin{equation}\label{eqiie2}
	\begin{split}
		\frac{\mathrm{d}\hat{i}}{\mathrm{d}t}&=-\frac{\gamma}{\lambda}\frac{\hat{i}}{\hat{i}_E}+1,\\
		\frac{\mathrm{d}\hat{i}_E}{\mathrm{d}t}&=F-\frac{\hat{i}_E}{\hat{s}}+\frac{\omega \alpha}{\lambda}\hat{i}.
	\end{split}
\end{equation}
Using the equation for $\hat{i}'(t)$   and the fact $\hat{i}(0)=0$ (initially only one vertex is infected),  
\begin{equation*}
	\hat{i}(t)=-\int_0^t \frac{\gamma}{\lambda}\frac{\hat{i}(u)}{\hat{i}_E(u)}\mathrm{d}u+t\geq 0.
\end{equation*}
Hence, we have, for $t\leq t_*$, 
\begin{equation}\label{intratio}
	\int_0^t \frac{\hat{i}(u)}{\hat{i}_E(u)}\mathrm{d}u\leq \frac{\lambda}{\gamma}t. 
\end{equation}
Since $\abs{\hat{i}'_E}$ and $\abs{F'(t)}$ are uniformly bounded for all $\lambda$ in a neighborhood of $\lambda_c$ and  $t\leq t_*$, there exists a constant $C$ such that $\hat{i}_E(t)\leq Ct$ and $F(t)\leq F(0)+Ct$.  Using this fact together with \eqref{eqiie2} and \eqref{intratio}, we see that for some constant $C_{\ref{ie1}}$,
\begin{equation}\label{ie1}
	\begin{split}
		\hat{i}_E(t)&\leq \int_0^t F(u)\mathrm{d}u+\int_0^t \frac{\omega \alpha }{\lambda}\hat{i}(u)\mathrm{d}u,\\
		&\leq \int_0^t (F(0)+Cu) \mathrm{d}u+  
		\frac{\omega \alpha }{\lambda} \int_0^t \frac{\hat{i}(u)}{\hat{i}_E(u)} Cu \mathrm{d}u,\\
		&\leq   F(0)t+  C_{\ref{ie1}} t^2. 
	\end{split}
\end{equation}

Using the expression of $F(t)$ and \eqref{f't1}, 
\begin{equation}\label{Fdata}
	\begin{split}
		F(0)&=-\frac{1+\gamma+\omega}{\lambda}+\mu,\\
		F'(0)&=-\mu+2\frac{\omega \alpha}{\lambda},\\
		F''(t)&=\frac{-2\omega \alpha }{\lambda (1-t)}.
	\end{split}
\end{equation}
Note that $F(0)\to 0$ as $\lambda\searrow\lambda_c$, while $F''(t)$ is bounded from above by $-2\omega \alpha/\lambda$.

We now divide the proof  of Theorem \ref{phase} into two cases. 

\subsection{Proof of Case \textcircled{1}}

Suppose either 
$\omega(2\alpha-1)\leq \gamma$ or 
\begin{equation*}
  	\omega(2\alpha-1)> \gamma \mbox{ and } \mu<2\omega \alpha/(\omega(2\alpha-1)-\gamma).
\end{equation*} 
Then by \eqref{f't2}, 
\begin{equation*}
  	-\mu+\frac{2\omega \alpha}{\lambda_c}=-\mu+\frac{2\omega \alpha (\mu-1)}{\gamma+\omega}<0.
\end{equation*}
By \eqref{Fdata} and \eqref{f't2}, 
\begin{equation*}
	\lim_{\lambda \to \lambda_c}F'(0)=-\mu+\frac{2\omega \alpha}{\lambda_c}<0.
\end{equation*}
Hence, there exists $\ep_0,\delta_0>0$, such that
for $\lambda-\lambda_c<\ep_0$ and $t<\delta_0$, 
\begin{equation}\label{f't6}
	F(t)\leq F(0)-C_{\ref{f't6}}t,
\end{equation}
for some $C_{\ref{f't6}}>0$.
By \eqref{intratio}, \eqref{ie1} and \eqref{f't6}, for $t\leq \delta_0$,
\begin{equation*}
	\begin{split}
		\hat{i}_E(t)&\leq \int_0^t F(u)\mathrm{d}u+\int_0^t \frac{\omega \alpha}{\lambda}\hat{i}(u)\mathrm{d}u\\
		&\leq \int_0^t (F(0)-C_{\ref{f't6}}u)\mathrm{d}u+
		\frac{\omega \alpha}{\lambda}\int_0^t \frac{\hat{i}(u)}{\hat{i}_E(u)}(F(0)u+C_{\ref{ie1}}u^2)\mathrm{d}u\\
		&\leq F(0)t-\frac{C_{\ref{f't6}}}{2}t^2+\frac{\omega \alpha}{\lambda}(F(0)t+C_{\ref{ie1}}t^2)\frac{\lambda }{\gamma}t\\
		&=F(0)t-\left(\frac{C_{\ref{f't6}}}{2}-\frac{\omega \alpha}{\gamma}F(0)\right)t^2+\frac{\omega \alpha}{\gamma}t^3,
	\end{split}
\end{equation*}
where we have used $\hat{i}(u)=(\hat{i}(u)/\hat{i}_E(u))\hat{i}_E(u)$ in the second inequality. 
Note that 
\begin{equation*}
  	\lim_{\lambda\searrow\lambda_c}F(0)=0 \mbox{ and } \lim_{\lambda\to \lambda_c}\left(\frac{C_{\ref{f't6}}}{2}-\frac{\omega \alpha}{\gamma}F(0)\right)=\frac{C_{\ref{f't6}}}{2}>0.
\end{equation*}
Consequently, there exist $\delta_1<\delta_0$ and $\ep_1<\ep_0$ such that for $t\leq \delta_1$ and $\lambda<\lambda_c+\ep_1$,
\begin{equation*}
	\hat{i}_E(t)\leq F(0)t-\left(\frac{C_{\ref{f't6}}}{2}-\frac{\omega \alpha}{\gamma}F(0)\right)t^2+\frac{\omega \alpha}{\gamma}t^3\leq 
	F(0)t-\frac{C_{\ref{f't6}}}{4}t^2.
\end{equation*}
From this we obtain that
\begin{equation*}
  	t_*=\inf\{0<t\leq \hat{s}(0): \hat{i}_E(t)=0\}\leq \frac{4F(0)}{C_{\ref{f't6}}},
\end{equation*}
which converges to 0 as $\lambda\searrow\lambda_c$ since $F(0)\to 0$. 

\subsection{Proof of Case \textcircled{2}}

Suppose $\omega(2\alpha-1)> \gamma$ and $\mu=2\omega \alpha/(\omega(2\alpha-1)-\gamma)$. Note that $\alpha$ must be positive in this case. By \eqref{Fdata} and \eqref{f't2}, we have
\begin{equation*}
	F'(0)<0 \mbox{ and }  \lim_{\lambda \to \lambda_c}F'(0)=-\mu+\frac{2\omega \alpha}{\lambda_c}=0.
\end{equation*}
Thus, by the third equation in \eqref{Fdata}, 
\begin{equation*}
  	F(t)\leq F(0)-\frac{\omega \alpha}{\lambda}t^2.
\end{equation*}
Using \eqref{ie1}, for some constant $C_{\ref{iebd6}}$ (independent of $\lambda-\lambda_c$),
\begin{equation}\label{iebd6}
	\begin{split}
		\hat{i}_E(t)&\leq \int_0^t F(u)\mathrm{d}u+\int_0^t \frac{\omega \alpha}{\lambda}\hat{i}(u)\mathrm{d}u\\
		&\leq \int_0^t \left(F(0)- \frac{\omega \alpha}{\lambda} u^2\right)\mathrm{d}u + \frac{\omega \alpha}{\lambda}\int_0^t \frac{\hat{i}(u)}{\hat{i}_E(u)}(F(0)u+C_{\ref{ie1}}u^2)\mathrm{d}u\\
		&\leq F(0)t+C_{\ref{iebd6}}(F(0)t^2+t^3), 
	\end{split}
\end{equation}
where we have again used $\hat{i}(u)=(\hat{i}(u)/\hat{i}_E(u))\hat{i}_E(u)$ in the second inequality. Substituting this bound into \eqref{iebd6} again, we see that
\begin{equation*}
	\begin{split}
		\hat{i}_E(t)&\leq 
		\int_0^t \left(F(0)- \frac{\omega \alpha}{\lambda} u^2\right)\mathrm{d}u+\frac{\omega \alpha}{\gamma} t\left(F(0)t+C_{\ref{iebd6}}(F(0)t^2+t^3)\right)\\
		& = F(0)t +\frac{\omega \alpha}{\gamma}F(0)t^2 +\left(-\frac{\omega \alpha}{3\lambda}+\frac{\omega \alpha}{\gamma}C_{\ref{iebd6}} F(0) \right) t^3 + \frac{\omega \alpha}{\gamma}C_{\ref{iebd6}}t^4.
	\end{split}
\end{equation*}
Since 
\begin{equation*}
  	\lim_{\lambda\searrow\lambda_c}F(0)=0 \mbox{ and } \lim_{\lambda \searrow\lambda_c} \left(-\frac{\omega \alpha}{3\lambda}+\frac{\omega \alpha}{\lambda}C_{\ref{iebd6}}F(0)\right)=-\frac{\omega \alpha}{3\lambda_c}<0,
\end{equation*}
we see that there exist  $\ep_2>0$ and $\delta_2>0$ such that for $\lambda<\lambda_c+\ep_2$ and $t\leq \delta_2$,
\begin{equation}\label{iebd7}
	\hat{i}_E(t)\leq 2F(0)t-\frac{\omega \alpha}{6\lambda_c}t^3.
\end{equation}
It follows from \eqref{iebd7} that
\begin{equation*}
  	t_*=\inf\{0<t\leq \hat{s}(0): \hat{i}_E(t)=0\}\leq \sqrt{\frac{12F(0)\lambda_c}{\omega \alpha}},
\end{equation*}
which converges to 0 as $\lambda\searrow\lambda_c$ since $F(0)\to 0$.
 
\appendix

\section{Proof of \eqref{m4bd}}\label{sec:ap1}

Let $N_i$ be the number of edges added to vertex $i$ when $i$ first becomes infected and $\hat{N}_i$  the number of infected edges of  $i$ just before it becomes recovered. The quadratic variation of $M^{(n)}_4(t)$ can be bounded by
\begin{equation}\label{m4bd-2}
	\langle M^{(n)}_4(t), M^{(n)}_4(t) \rangle = \langle \hat{I}^{(n)}_E(t), \hat{I}^{(n)}_E(t) \rangle \leq \sum_{i=1}^n (\hat{N}_{i}^2+N_i^2) + Q(t)+\sum_{j=\hat{S}^{(n)}(t)}^n Z_j^2,
\end{equation}
where $Q(t)$ is the  number of arrivals of rewiring/dropping Poisson processes by time $t$, and  $Z_j$ is distributed as 1+Binomial($2\mu n$, $1/j$). In addition, $Z_j,1\leq j\leq n$ are independent. By \eqref{boundqt},
\begin{equation}\label{m4bd-3}
	\E(Q(t))\leq \frac{\gamma n t}{\lambda}.
\end{equation} 
By Lemma \ref{boundni}, for some constant $C>0$,
\begin{equation}\label{m4bd-4}
	\E\left(\sum_{i=1}^n (\hat{N}_{i}^2+{N}_{i}^2) \right) \leq Cn.
\end{equation}
Define the event 
\begin{equation}\label{m100}
	\Omega_{\ref{m100}}=\left\{\sup_{t\leq \Hat{\tau}} \abs{M^{(n)}_1(t)}\leq (\hat{s}(0)-t_*)n/4\right\}.
\end{equation}
By \eqref{quast} and the Burkholder-Davis-Gundy inequality (see, e.g., \cite[Theorem 7.34]{MR2933773} with $p=6$), we see that
\begin{equation*}
	\E\left(\sup_{t\leq \hat{\tau}^{(n)}} M^{(n)}_1(t)^6 \right)\leq Cn^3,
\end{equation*}
and consequently,
\begin{equation}\label{m101}
	\P(\Omega_{\ref{m100}}^c)\leq \frac{C n^3}{((\hat{s}(0)-t_*)n/4)^6}\leq \frac{C'}{n^4}.
\end{equation}
On the event $\Omega_{\ref{m100}}$, by \eqref{hats2} and \eqref{m101}, 
\begin{equation*}
	\inf_{t\leq  \hat{\tau}^{(n)} \wedge ((\hat{s}(0)+t_*)/2)} \hat{S}^{(n)}(t) \geq \left(\hat{s}(0)-\frac{t_*+\hat{s}(0)}{2}\right)n-\frac{\hat{s}(0)-t_*}{4}n\geq \frac{\hat{s}(0)-t_*}{4}n.
\end{equation*}
Define 
\begin{equation}\label{defz0}
	Z=\sum_{j=1}^n Z_j^2.
\end{equation}
We have 
\begin{equation}
	\sum_{j=\hat{S}^{(n)}(t)}^n Z_j^2 \leq Z\mathbf{1}_{\Omega^c_{\ref{m100}}}+\sum_{j=(\hat{s}(0)-t_*)n/4}^n Z_j^2. 
\end{equation}
It follows that 
\begin{equation}\label{m4bd-10}
	\E\left(   \sum_{j=\hat{S}^{(n)}(t)}^n Z_j^2 \right) \leq \E\left(Z\mathbf{1}_{\Omega^c_{\ref{m100}}}\right) +\E\left(\sum_{j=(\hat{s}(0)-t_*)n/4}^n Z_j^2 \right).
\end{equation}
Observe that the second moment of Binomial($2\mu n$, $1/j$) is bounded by 
\begin{equation}\label{bi2bd}
	\frac{2\mu n}{j}+	\left(\frac{2\mu n}{j}\right)^2 \leq \frac{C n^2}{j^2}.
\end{equation}
Thus,
\begin{equation}\label{m4bd-9}
	\E\left(\sum_{j=(\hat{s}(0)-t_*)n/4}^n Z_j^2 \right)\leq C\sum_{j=(\hat{s}(0)-t_*)n/4}^n\frac{n^2}{j^2}\leq C'n.
\end{equation}
To control $Z$, similarly to \eqref{bi2bd}, the fourth moment of 	 Binomial($2\mu n$, $1/j$) can be bounded by $C(n/j)^4$. Therefore,
\begin{equation*}
	\E(Z^2)\leq Cn\E\left(\sum_{j=1}^n \frac{n^4}{j^4}\right)\leq  C'n^5.
\end{equation*}
By Cauchy-Schwartz inequality, we get
\begin{equation}\label{m4bd-8}
	\E\left(Z\mathbf{1}_{\Omega^c_{\ref{m100}}}\right)\leq \sqrt{\E(Z^2)\P(\Omega^c_{\ref{m100}})}\leq C\sqrt{n^5 \cdot n^{-3}}= Cn.
\end{equation}
Combining \eqref{m4bd-10}, \eqref{m4bd-9} and \eqref{m4bd-8}, we see that
\begin{equation}\label{m4bd-7}
	\E\left(\sum_{j=\hat{S}^{(n)}(t)}^n Z_j^2 \right)\leq Cn.
\end{equation}
Equation \eqref{m4bd} now follows from \eqref{m4bd-2}, \eqref{m4bd-3}, \eqref{m4bd-4}, \eqref{m4bd-7} and  $L^2$ maximal inequality for martingales.  

\section{Proof of \eqref{eq:haty}}\label{app:haty}

Similarly to the proof of \eqref{eq:hatsw}, we divide the proof of \eqref{eq:haty} into two steps. For $t\geq \hat{\tau}^{(n)}_{\delta}$ we set $\hat{\mathbf{Y}}^{(n)}(t)$ to be $\hat{\mathbf{Y}}(\hat{\tau}^{(n)}_{\delta})$. We first	show that $\hat{\mathbf{Y}}^{(n)}(\cdot)/n$ is a tight sequence in D[0, $t_*$]. Since $\hat{\mathbf{Y}}^{(n)}(0)/n \to \hat{\mathbf{y}}^{(n)}(0)/n$, we know  $\{ \hat{\mathbf{Y}}^{(n)}(0)/n \}_{n\geq 1}$ is a tight sequence of random variables. To establish tightness of $\{ \hat{\mathbf{Y}}^{(n)}(t)/n, 0\leq t\leq t_*\}_{n\geq 1}$  we need to show that for any fixed $\ep_1,\ep_2>0$, there exist $\theta>0$ and $n_0>0$ so that for $n\geq n_0$,
\begin{equation}\label{ytightness}
	\P\left (\sup_{\abs{t_1-t_2}\leq \theta, t_1<t_2\leq t_* \wedge \hat{\tau}^{(n)}_{\delta}   } \norm{ \frac{\hat{\mathbf{Y}}^{(n)}(t_1) }{n} - \frac{\hat{\mathbf{Y}}^{(n)}(t_2)}{n} }  \geq \ep_2 \right) \leq \ep_1.
\end{equation}

Assuming \eqref{ytightness} for the moment, we show that the limit of $\hat{\mathbf{Y}}^{(n)}(t)/n$ coincides with $\hat{\mathbf{y}}(t)$. Similarly to the proof that the  limit of $\hat{W}^{(n)}(t)/n$ coincides with $\hat{w}(t)$, we would like to show that any sequential limit point $\mathbf{\check{y}}(t)$ coincides with $\hat{\mathbf{y}}(t)$, which is then necessarily unique. This is clearly the case since, after dividing both sides of \eqref{hatieq}, \eqref{hatiEeq} by $n$ and then sending $n$ to $\infty$, we obtain that
\begin{align*}
	\check{i}(t)&=\hat{i}(0)-\int_0^t\frac{\gamma\check{i}(u)}{\lambda\hat{i}_E(u)}\mathrm{d}u+t, \\
	\check{i}_E(t)&=\hat{i}_E(0)+\int_0^t \left(\mu\hat{s}(u)+\frac{2\hat{w}(u)}{\hat{s}(u)}- \frac{\check{i}_E(u)}{\hat{s}(u)}+ \frac{\omega \alpha}{\lambda}\check{i}(u) \right)\mathrm{d}u - \frac{\lambda+\omega+\gamma}{\lambda}t.
\end{align*}
For $t\leq t_*$, this equation has a unique solution $\mathrm{\check{y}}(t)=\hat{\mathbf{y}}(t)$, as proved in Theorem \ref{exiuni}. 

It remains to establish \eqref{ytightness}. Using \eqref{hatieq} and \eqref{hatiEeq}, 
\begin{equation*}
	\frac{\hat{I}^{(n)}(t_2)}{n}-\frac{\hat{I}^{(n)}(t_1)}{n}=\int_{t_1}^{t_2} \left(1-\frac{\gamma\hat{I}^{(n)}(u)}{\lambda\hat{I}^{(n)}_E(u)} \right)\mathrm{d}u +\frac{M_3^{(n)}(t_2)-M_3^{(n)}(t_1)}{n},
\end{equation*}
and
\begin{equation}\label{iEdiff}
  \begin{split}
    \frac{\hat{I}^{(n)}_E(t_2)}{n}-\frac{\hat{I}^{(n)}_E(t_1)}{n}=&
    \int_{t_1}^{t_2} \left(\frac{\hat{S}^{(n)}(u)-1}{n}\mu + \frac{2\hat{W}^{(n)}(u)}{\hat{S}^{(n)}(u)} -\frac{\hat{I}^{(n)}_E(u)-1}{\hat{S}^{(n)}(u)} +
    \frac{\omega\alpha\hat{I}^{(n)}(u)}{\lambda n} -\frac{\lambda+\omega+\gamma}{\lambda} \right) \mathrm{d}u\\
    &+ \frac{M^{(n)}_4(t_2)-M^{(n)}_4(t_1)}{n}.
  \end{split}
\end{equation} 
We define the event
\begin{equation}\label{iie}
	\Omega_{\ref{iie}}=\{\hat{S}^{(n)}(t)\geq \frac{\hat{s}(0)-t_*}{2}n, \quad  \forall \ t\leq t_*\wedge \hat{\tau}^{(n)} \}\cap \{\sup_{t\leq \hat{\tau}^{(n)}\wedge t_*} \left(\abs{M^{(n)}_3(t)} + \abs{M^{(n)}_4(t)}\right) \leq n^{2/3}\}.
\end{equation}
By \eqref{slim2}, \eqref{m3bd} and \eqref{m4bd}, 
\begin{equation*}
	\lim_{n\to\infty}\P(\Omega_{\ref{iie}})=1.
\end{equation*}
On $\Omega_{\ref{iie}}$, for $t_1<t_2\leq \hat{\tau}^{(n)}_{\delta}\wedge t_*$,
\begin{equation}\label{it2-t1}
	\begin{split}
    	\abs{   \frac{\hat{I}^{(n)}(t_2)}{n}-\frac{\hat{I}^{(n)}(t_1)}{n}}\leq (t_2-t_1)\left(1+\frac{2\gamma }{\lambda \delta}\right)+\frac{2n^{2/3}}{n}\leq C(t_2-t_1)+2n^{-1/3},
  	\end{split}
\end{equation}
and 
\begin{equation}\label{iet2-t2}
  	\begin{split}
    	&\frac{\hat{I}^{(n)}_E(t_2)}{n}-\frac{\hat{I}^{(n)}_E(t_1)}{n}\\
    	\leq & (t_2-t_1)\left(\mu+\frac{4 \mu}{\hat{s}(0)-t_*}+\frac{4\mu}{\hat{s}(0)-t_*}+\frac{\omega \alpha }{\lambda }+\frac{\lambda +\omega+\lambda }{\lambda }  \right)+\frac{2n^{2/3}}{n}\\
    	\leq &C(t_2-t_1)+2n^{-1/3}.
  \end{split}
\end{equation}
Equation \eqref{ytightness} now follows from \eqref{it2-t1}, \eqref{iet2-t2} and the fact that
\begin{equation*}
  	\begin{split}
    	\norm{\hat{\mathbf{Y}}^{(n)}(t_2)-\hat{\mathbf{Y}}^{(n)}(t_1)}
    	= &\sqrt{\abs{\hat{I}^{(n)}(t_2)-\hat{I}^{(n)}(t_1)}^2+\abs{\hat{I}^{(n)}_E(t_2)-\hat{I}^{(n)}_E(t_1)}^2}\\
    	\leq &\abs{\hat{I}^{(n)}(t_2)-\hat{I}^{(n)}(t_1)}+\abs{\hat{I}^{(n)}_E(t_2)-\hat{I}^{(n)}_E(t_1)}.
  	\end{split}
\end{equation*}
We have now completed the proof of \eqref{eq:haty}. 

\section{Proof of Lemma \ref{lem:i=0}}\label{sec:ap2}

By \eqref{hati_E}, we see that $\abs{\hat{i}'_E(t)}$ is uniformly bounded. Hence, there exists $C>0$ such that 
\begin{equation}\label{ietub}
  	\hat{i}_E(t)\leq C(t_*-t)
\end{equation}
for all $t\leq t_*$. Assuming $\limsup_{t\to t_*}\hat{i}(t)=a>0$, then by the fact that $\hat{i}'(t)\leq 1$, there exists an $\ep>0$ such that $\hat{i}(t)>a/2$ for $t\in (t_*-\ep,t_*)$. By \eqref{hati} and \eqref{ietub},
\begin{equation*}
  	\hat{i}(t)-\hat{i}(t_*-\ep)= -\int_{t_*-\ep}^{t} \frac{\hat{i}(u)}{\hat{i}_E(u)}\mathrm{d}u  +(t-t_*+\ep)\leq -\int_{t_*-\ep}^t \frac{a/2}{C(t_*-u)}\mathrm{d}u+(t-t_*+\ep),
\end{equation*}
which goes to $-\infty$ as $t\to t_*$. This is a contradiction. 
Hence, we must have $\lim_{t\to t_*}\hat{i}(t)=0$, proving Lemma \ref{lem:i=0}. 

\section{Proof of Corollary \ref{cor:lbtau}}\label{sec:ap3}

From the definition of $\hat{\tau}^{(n)}$ (the first time when $\hat{I}_E^{(n)}$ reaches 0), we see that 
\begin{equation*}
  	\hat{\tau}^{(n)} \geq \hat{\tau}^{(n)}_{\delta},\quad \forall \ \delta>0.
\end{equation*}
By Lemma \ref{lem:taudelta}, for any $\ep>0$, we can find a $\delta>0$, so that 
\begin{equation*}
  	\lim_{n\to\infty}\P(\hat{\tau}^{(n)}_{\delta}\geq t_*-\ep)=1.
\end{equation*}
It follows that
\begin{equation*}
  	\lim_{n\to\infty}\P(\hat{\tau}^{(n)} \geq t_*-\ep)\geq \lim_{n\to\infty}\P(\hat{\tau}^{(n)}_{\delta}\geq t_*-\ep)=1.
\end{equation*}
Equation \eqref{sizelb} follows from the convergence of $\hat{S}^{(n)}(t)/n$ to $\hat{s}(t)(=\hat{s}(0)-t)$. Indeed, by \eqref{eq:lbtau} and \eqref{hats},
\begin{equation*}
  	\P(\hat{S}^{(n)}(t_*-\ep) \leq \hat{s}(t_*-\ep)+\ep, \hat{\tau}^{(n)}>t_*-\ep)=1.
\end{equation*}
Since
\begin{equation*}
  	\hat{s}(t_*-\ep)+\ep=\hat{s}(0)-(t_*-\ep)+\ep=\hat{s}(0)-t_*+2\ep,
\end{equation*}
and  
\begin{equation*}
  	\Lambda^{(n)}=n-\hat{S}^{(n)}(\hat{\tau}^{(n)}) \geq n-\hat{S}^{(n)}(t_*-\ep),
\end{equation*} 
we obtain that
\begin{equation*}
  	\lim_{n\to\infty}\P\left(\Lambda^{(n)}\geq (1-\hat{s}(0)+t_*-2\ep)n \right)=1.
\end{equation*}
Equation \eqref{sizelb} then follows since $\ep>0$ is arbitrary.

\section*{Acknowledgements}

We thank Rick Durrett for suggesting this problem. 

% The author would like to thank an anonymous referee for pointing out several mistakes in a preliminary version.

\bibliographystyle{alea3}
\bibliography{ref}
\end{document}